\documentclass[11pt]{amsart}
\usepackage{amsfonts,graphicx,
stmaryrd,amscd,latexsym,amsbsy}
\usepackage{pst-node}
\usepackage{tikz-cd} 
\usepackage{subcaption}
\usepackage[colorlinks=true]{hyperref}

\usepackage{pst-node}
\usepackage{pst-intersect}

\usepackage{amssymb}
\usepackage{amsmath,amsthm}
\usepackage{xy}
\usepackage[tmargin=1in]{geometry}

\numberwithin{equation}{section}

\newtheorem{theorem}{Theorem}[section]

\newtheorem{lemma}[theorem]{Lemma}
\newtheorem{definition}[theorem]{Definition}
\newtheorem{proposition}[theorem]{Proposition}
\newtheorem{corollary}[theorem]{Corollary}

\newtheorem{example}[theorem]{Example}

\theoremstyle{remark}
\newtheorem{remark}[theorem]{Remark}

\begin{document}
\title{Affine Hecke algebras and symmetric quasi-polynomial duality}
\author{Vidya Venkateswaran}
\address{V. Venkateswaran, Center for Communications Research, 805 Bunn Dr, Princeton,
NJ 08540, USA. }
\email{vidyavenkat52@gmail.com}
%%%%%%%%%%%%%%%%%%%%%%%%%%
\begin{abstract}
In a recent paper \cite{SSV2}, we introduced quasi-polynomial generalizations of Macdonald polynomials for arbitrary root systems via a new class of representations of the double affine Hecke algebra.  These objects depend on a deformation parameter $q$, Hecke parameters, and an additional torus parameter.  In this paper, we study \textit{antisymmetric} and \textit{symmetric} quasi-polynomial analogs of Macdonald polynomials in the $q \rightarrow \infty$ limit.  We provide explicit decomposition formulas for these objects in terms of classical Demazure-Lusztig operators and partial symmetrizers, and relate them to Macdonald polynomials with prescribed symmetry in the same limit.  We also provide a complete characterization of (anti-)symmetric quasi-polynomials in terms of partially (anti-)symmetric polynomials.  
As an application, we obtain formulas for metaplectic spherical Whittaker functions associated to arbitrary root systems.  For $GL_{r}$, this recovers some recent results of Brubaker, Buciumas, Bump, and Gustafsson, and proves a precise statement of their conjecture about a ``parahoric-metaplectic" duality.
\end{abstract}
\maketitle
%\tableofcontents

\noindent
 %%%%%%%%%%%%%%%%%%%%%%%%
 \section{Introduction}\label{Intro}

 In a recent paper \cite{BBBG21}, Brubaker, Buciumas, Bump and Gustafsson prove a striking new duality result about metaplectic spherical Whittaker functions for $GL_{r}$.  These objects, which can be defined for arbitrary type, are matrix coefficients of representations of $n$-fold covers of $p$-adic groups, as well as ``p-parts" of Weyl group multiple Dirichlet series, and have been the subject of much investigation over the last two decades (see \cite{Bu} for an overview).  Much recent work has centered around finding combinatorial formulas for metaplectic Whittaker functions, as well as interpreting these objects using representations of Weyl groups and Hecke algebras (see e.g., \cite{CG07, CG, CGP, PP, SSV, BBBG20, BBBG21}).  Of particular interest are connections to the theory of Macdonald polynomials (see \cite{C, Cm, Ch, Ha, Ma} for background on Macdonald polynomials).  In the non-metaplectic case (i.e., the special case $n=1$), Whittaker functions can be understood in terms of the well-known polynomial representation of Hecke algebras as well as degenerations of Macdonald polynomials and related objects, see Table 2 of \cite{BBBG19}.
 
 One recent approach to studying metaplectic Whittaker functions has been through the study of classes of solvable lattice models that have these functions as their partition functions \cite{BBBG19, BBBG20, BBBG21}.  Using this approach, these authors discovered a key exchangeability feature of their lattice models, which then enabled them to provide formulas for certain classes of $GL_r$ \textit{metaplectic} spherical Whittaker functions in terms of \textit{non-metaplectic} Whittaker functions; they call this parallel ``duality".  This surprising discovery was the motivation for the present work, in which we generalize these results in several directions.
 
More precisely,  fix a metaplectic parameter $n \in \mathbb{Z}_{\geq 1}$ and rank $r-1 \in \mathbb{Z}_{\geq 1}$.   Let $F$ be a local non-archimedean field with uniformizer $\varpi$.  Let $P = \mathbb{Z}^{r}$ and $P^{+} = \{ (\lambda_1, \dots, \lambda_r) : \lambda_i \geq \lambda_{i+1} \text{ for all } i \} \subseteq P$ denote the set of $GL$-weights and partitions, respectively.  Let $\varpi^{\lambda}$, for $\lambda = (\lambda_1, \dots, \lambda_r) \in P$, denote the diagonal element in $G= GL_{r}(\mathbb{F})$ with entries $\varpi^{\lambda_1}, \dots, \varpi^{\lambda_r}$.  Let $\mathbf{y} = (y_1, \dots, y_r) \in (\mathbb{C}^{\times})^{r}$ and $g \in G$.

There is a metaplectic spherical Whittaker function, denoted $\widetilde{\mathcal{W}}^{m} = \widetilde{\mathcal{W}}^{m}(\mathbf{y}; g;v)$, that has been the subject of much study and applications (see e.g., \cite{CO, CG07, CG, CGP, PP, PP2, BBBG19, BBBG20, BBBG21}, and references therein).  It is a polynomial in $\mathbf{y}$, depends on a parameter $v$ and is determined by its values on the arguments $g = \varpi^{\rho_{GL} - \mu}$, where $\rho_{\text{GL}} = (r-1, r-2, \dots, 0)$ and $\mu - \rho_{GL} \in P^{+}$.  In \cite{BBBG21}, the authors introduced the idea of studying certain components $\widetilde{\phi}^{o}_{\theta}$ of $\widetilde{\mathcal{W}}^{m}$, with $\theta \in (\mathbb{Z} / n \mathbb{Z})^{r}$; these satisfy
\begin{equation}\label{eq:W_theta}
\widetilde{\mathcal{W}}^{m} = \sum_{\theta \in (\mathbb{Z} / n \mathbb{Z})^{r} } \widetilde{\phi}_{\theta}^{o}.
\end{equation}
In Theorems D and E of \cite{BBBG21}, formulas for $\widetilde{\phi}_{\theta}^{o}(\mathbf{y};\varpi^{\rho - \mu}; v)$ are provided in terms of \textit{non-metaplectic} Whittaker functions, when the indexing element $\theta \in (\mathbb{Z} / n \mathbb{Z})^{r}$ has either all parts distinct, or all parts equal.  They also express $\widetilde{\phi}_{\theta}^{o}(\mathbf{y};\varpi^{\rho - \mu}; v)$ in these two cases using (non-metaplectic) Hecke operators $\mathcal{T}_{w,v}$.  In the introduction and Remark 4.11 of the same paper, the authors conjecture that, for arbitrary $\theta$, $\widetilde{\phi}_{\theta}^{o}$ should be related to (non-metaplectic) \textit{parahoric} Whittaker functions.  

One of our main results is the following theorem, which proves a precise statement of their conjecture:

\begin{theorem}\label{cor:Whitt_typeA}

Set $v = q^{2}$ and $W = S_{r}$.  Let $\theta \in (\mathbb{Z} / n \mathbb{Z})^{r}$ and $\mu - \rho_{GL} \in P^{+}$.  Then $\mathbf{y}^{\rho_{\text{GL}} - \lfloor \theta \rfloor_{n}} \tilde{\phi}_{\theta}^{o}(\mathbf{y}; \varpi^{\rho_{\text{GL}} - \mu})$ is a polynomial in $\mathbf{y}^{n}$, and it vanishes unless $\mu \text{ mod }n$ is a permutation of $\theta$, and in this case
\begin{align*}
w_0 \Big( \mathbf{y}^{\frac{\rho_{\text{GL}} - \lfloor \theta \rfloor_{n}}{n}} \tilde{\phi}_{\theta}^{o}(\mathbf{y}^{1/n}; \varpi^{\rho_{\text{GL}} - \mu} ; v) \Big)& = C \cdot \mathcal{T}_{w, v^{1/2}}\sum_{u \in W_{ \textbf{c}}} \mathcal{T}_{u, v^{1/2}} \mathcal{T}_{w', v^{1/2}}^{-1} \mathbf{y}^{\lambda + \rho_{\text{GL}}} \\
&= C'  \cdot \mathbf{y}^{\rho_{GL}} \psi_{w^{\textbf{c}}}^{J_{\textbf{c}}}(\mathbf{y}; \varpi^{-\lambda}w'; v^{-1}), 
\end{align*}
with $\lambda = \lfloor \frac{\mu}{n} \rfloor - \rho_{\text{GL}}$.  Here $w,w' \in W$ and $\mathbf{c}$ is a decreasing $r$-tuple of elements in $\{1, \dots, n \}$; these are determined uniquely by $-\mu = w' \textbf{c} \text{ mod } n$ and $-w_0 \theta = w \textbf{c} \text{ mod }n$ with $w', w_0 w \in W^{\mathbf{c}}$.  Also $J_{\textbf{c}} = \{1 \leq j \leq r-1 : c_j = c_{j+1} \}$.  Finally the coefficients are given by
\begin{align*}
C &=  \frac{(-1)^{l(w_0)} v^{l(w)/2 - l( w')/2 + l(w_0)/2 + l(w_0(W_{\mathbf{c}} ))} h^{m,  \mathbf{c}}_{v}(w_0 w)   }{h^{m,  \mathbf{c} }_{v}( w' )} \\
C' &=  v^{l(w') - l(w_{\textbf{c}})} C.
\end{align*}
\end{theorem}
Here, $\tilde{\phi}_{\theta}^{o} := \sum_{w \in W} \phi_{\rho_{GL} - \lfloor \theta \rfloor_{n}, w}$, where $\lfloor \theta \rfloor_{n}$ denotes taking the least nonnegative integer residue modulo $n$ elementwise, and $\phi_{\rho_{GL} - \lfloor \theta \rfloor_{n},w}$ are metaplectic Iwahori-Whittaker functions as defined in Section 4.1 of \cite{BBBG20}.  The Hecke operators $\mathcal{T}_{w,v}$ are as in Definition \ref{def:DW_opers}, $\psi^{J}_{w^{c}}$ are (non-metaplectic) parahoric Whittaker functions as introduced and studied in \cite{BBBG19}, and the coefficients $h_{v}^{m, \mathbf{c}}(w)$ are given explicitly in Definition \ref{def:met_h_stats} and \ref{def:met_specialization}.  Moreover, $W_{\mathbf{c}} =  \langle s_{j} : j \in J_{\mathbf{c}}\rangle$ is the parabolic stabilizer subgroup of $\mathbf{c}$ inside $W$, $W^{\mathbf{c}}$ is the set of minimal length representatives of $W / W_{\mathbf{c}}$ and $w = w^{\mathbf{c}} w_{\mathbf{c}}$ with $w^{\mathbf{c}} \in W^{\mathbf{c}}$, $w_{\mathbf{c}} \in W_{\mathbf{c}}$.  The element $w_0 \in W$ is the longest word.  The proof of Theorem \ref{cor:Whitt_typeA} can be found in Section \ref{sec:met_results}.

We are also able to prove results analogous to Theorem \ref{cor:Whitt_typeA} for \textit{arbitrary} root system types and special choices of $\theta$ (see Theorem \ref{thm:Whitt_arb}) as well as for symmetric variants of metaplectic Whittaker functions (see Theorem \ref{thm:met_plus_main}).  Using \eqref{eq:W_theta}, a byproduct of our results is a new formula for the metaplectic spherical Whittaker function $\widetilde{\mathcal{W}}^{m}(\mathbf{z};\varpi^{\rho - \mu})$ for all $\mu$ for $GL_{r}$ and special values of $\mu$ for other types.  Our proofs are purely representation-theoretic, and, in particular, do not employ lattice models.  In \cite{A}, Amol Aggarwal outlined an argument which matches (components of) metaplectic spherical Whittaker functions to non-metaplectic parahoric Whittaker functions in the $GL_{r}$ case, through a comparison of the lattice models in \cite{ABW} and \cite{BBBG21}, along with the key ``color-merging'' result Theorem 5.2.2 of \cite{ABW}.  This provides another approach to establishing Theorem \ref{cor:Whitt_typeA}, and it is an interesting open question whether this technique can be extended to other types.

We also establish results for the $q \rightarrow \infty$ limits of the quasi-polynomial generalizations of (anti-)symmetric Macdonald polynomials introduced in \cite{SSV2} (see also \cite{SSV} and \cite{Sai}).  For simplicity, we will briefly describe our results in the $GL_{r}$ context; note however that the rest of the paper will deal with arbitrary root systems.  Fix $E = \mathbb{R}^{r}$, Weyl group $W = S_{r}$ generated by simple reflections $s_{i}$ for $1 \leq i \leq r-1$, and lattice $\Lambda = \mathbb{Z}^{r}$.  Let $C^{0}_\Lambda= \{ (y_1, \dots, y_r) \in E : 0 \leq y_i - y_{i+1} \leq 1 \text{ for } 1 \leq i \leq r-1 \text{ and } y_1 - y_r < 1 \} $.
Also, for $c \in C^{0}_\Lambda$, let $\widetilde{\mathcal{O}}_{c} = \{ \mu + wc : \mu \in \Lambda, w \in W \}$ denote the affine Weyl group orbit and $\mathbb{F}[\widetilde{\mathcal{O}}_{c}]$ the corresponding group algebra for a fixed field $\mathbb{F}$; the elements of $\mathbb{F}[\widetilde{\mathcal{O}}_{c}]$ are called quasi-polynomials.  
Let $\widetilde{\mathcal{H}}_{\Lambda}$ denote the $\Lambda$-extended affine Hecke algebra with Hecke parameter $t$ (see Definition \ref{def:AHA}).  

 \begin{theorem}[\cite{SSV2}]\label{thm:qp_rep_GL}
There is a \textit{quasi-polynomial representation} $\pi^{qp}$ of $\widetilde{\mathcal{H}}_{\Lambda}$ on $\mathbb{F}[\widetilde{\mathcal{O}}_{c}]$ given by the following explicit formulas:
\begin{equation*}
\pi^{qp}(T_{j}) x^{y} =  \begin{cases} 1, & y_j - y_{j+1} \in \mathbb{R} \setminus \mathbb{Z} \\
t, & y_j - y_{j+1} \in \mathbb{Z}
\end{cases} \Bigg \} \cdot x^{s_j y} + (t - t^{-1}) x^{y}\cdot \frac{\Big(\frac{x_j}{x_{j+1}}\Big)^{ - \lfloor y_{j} - y_{j+1} \rfloor } - 1 }{\frac{x_{j}}{x_{j+1}} - 1} 
\end{equation*}
for $y \in \widetilde{\mathcal{O}}_{c}$, $1 \leq j \leq r-1$, and $\pi^{qp}(x^{\lambda}) x^{y} = x^{\lambda + y}$
for $\lambda \in \Lambda$ and $y \in \widetilde{\mathcal{O}}_{c}$.
\end{theorem}
\noindent Crucially, for $c=0$, the formulas above recover the (classical) polynomial representation $\pi$ of $\widetilde{\mathcal{H}}_{\Lambda}$ on $\mathbb{F}[\Lambda]$ and in this case $\pi(T_i)$ are the Demazure-Lusztig operators.

In fact, as shown in \cite{SSV2}, given a torus character $\tau: \Lambda \rightarrow \mathbb{F}$ satisfying certain compatibility conditions with respect to $c$, the representation $\pi^{qp}$ can be extended to a representation of the double affine Hecke algebra $\mathbb{H}$.  This representation has a number of interesting properties, for example, a construction as a $Y$-induced cyclic module, a degeneration to Cherednik's basic representation when $c = 0$, and a connection to the theory of metaplectic Whittaker functions.
 
There are several interesting families of quasi-polynomials in $\mathbb{F}[ \widetilde{\mathcal{O}}_{c}]$ that can be described via the representation $\pi^{qp}$; we briefly describe some of the relevant results from \cite{SSV2}.  For every $c \in C^{0}_\Lambda$, there is a family of quasi-polynomials $E_{y} = E_{y}(q; \mathbf{t}; \tau) \in \mathbb{F}[\widetilde{\mathcal{O}}_{c}]$, indexed by $y \in  \widetilde{\mathcal{O}}_{c}$, and depending on a deformation parameter $q$, Hecke parameters $\mathbf{t}$, and an additional torus parameter $\tau$, that are simultaneous eigenfunctions of $\pi^{qp}(Y^{\lambda})$.  Here $\{Y^{\lambda}\}_{\lambda \in \Lambda} \subset \mathbb{H}$ are a certain family of commuting elements.  Moreover, the quasi-polynomials $E_{\lambda}$ for $\lambda \in \Lambda$ (i.e., $c = 0$) coincide with nonsymmetric Macdonald polynomials.  Another connection is to metaplectic Whittaker functions: for $c \in C_\Lambda^0$ and for $\tau$ satisfying certain regularity conditions, the $q \rightarrow \infty$ limiting case of these polynomials, denoted $\overline{E}_{y}$, can be identified with metaplectic Iwahori-Whittaker functions after a reparametrization which introduces Gauss sums.  One can also consider (anti-)symmetric variants $E_{y}^{\pm}$ by applying the Hecke algebra (anti-)symmetrizer $\pi^{qp}(\mathbf{1}^{\pm})$ to $E_{y}$, for $y \in \widetilde{\mathcal{O}}_{c}$.  The symmetric variants coincide with symmetric Macdonald polynomials for $c=0$ and $y \in \Lambda$.  Moreover, for $\tau$ satisfying certain regularity conditions, if one specializes the Hecke parameters in an appropriate way, the anti-symmetric variants for arbitrary $c$ in the limit $q \rightarrow \infty$ are metaplectic spherical Whittaker functions (up to an explicit reparametrization).  In this paper, we provide explicit formulas for the family of quasi-polynomials $p^{\pm}_{y} := \pi^{qp}(\mathbf{1}^{\pm}) x^{y}$, where $c \in C^{0}_\Lambda$ and $y \in \widetilde{\mathcal{O}}_{c}$.  For $y$ anti-dominant, these quasi-polynomials are equal to $\overline{E}_{y}^{\pm}$ and for $y$ dominant, they are related to $\overline{E}_{-y}^{\pm}(q; \mathbf{t}^{-1}; \tau)$, see Section \ref{sec:connections_to_pols}.
 
The starting point of this paper is the observation that, for any quasi-polynomial $f \in \mathbb{F}[\widetilde{\mathcal{O}}_{c}]$, there is a unique decomposition
 \begin{equation}\label{qp_gamma}
 f = \sum_{w \in W^{c}} \gamma_{w}(f) x^{w c},
 \end{equation}
 where $W_{c}$ is the stabilizer subgroup of $c$, $W^{c}$ is the set of minimal length representatives of $W / W_{c}$ and $\gamma_{w}: \mathbb{F}[\widetilde{\mathcal{O}}_{c}] \rightarrow \mathbb{F}[\Lambda]$ for $w \in W^{c}$ are coefficient functions (see Section \ref{sec:qp} for more details).  Note that $\gamma_{w}(f)$ are \textit{polynomials} in $\mathbb{F}[\Lambda]$ and computing $f$ can be boiled down to determining these coefficients.  In this paper, we explicitly compute these coefficients for certain families of quasi-polynomials.  We prove the following theorem in Section \ref{sec:QP_results}:

\begin{theorem}\label{cor:qp_main_GL}
Let $c \in C^{0}_\Lambda$, $y \in \widetilde{\mathcal{O}}_{c}$ and write (uniquely) $y = \mu + \hat{w} c$, where $\mu \in \Lambda$ and $\hat{w} \in W^{c}$.  Then for $w \in W^c$ we have
\begin{align*}
\gamma_{w}(p^{+}_{y}) &= t^{l(w_0)}  \cdot w_0 \pi \Big( T_{(w_0 w)^{-1}}^{-1} \mathbf{1}^{+}_{J_{c}} T_{\hat{w}^{-1}}\Big) x^{\mu}  \\
\gamma_{w}(p^{-}_{y}) &= (-1)^{l(\hat{w}) + l(w)} t^{2l(w_0(W_c)) - l(w_0)}  \cdot   \iota w_0 \pi \Big( T_{w_0 w} \mathbf{1}^{-}_{J_{c}} T_{\hat{w}}^{-1} \Big) \iota x^{\mu},
\end{align*}
where $J_c = \{ 1 \leq j \leq r-1 :  c_{j} = c_{j+1} \}$ and $w_0(W_c)$ is the longest word in $W_c$. Moreover, if $y$ is dominant, we have
\begin{align*}
\gamma_{w}( \iota \overline{E}_{-y}^{+}(q; t^{-1}; \tau)) &=  t^{-l(w_0)}  \cdot w_0 \pi (T_{(w_0 w)^{-1}}^{-1}) p_{\hat{w}^{-1} \mu}^{J_{c}, +} \\
\gamma_{w}( \iota \overline{E}_{-y}^{-}(q; t^{-1}; \tau)) &= (-1)^{l(\hat{w}) + l(w)} t^{2l(w_0(W_c)) + l(w_0)}  \cdot   \iota w_0 \pi ( T_{w_0 w}) p_{\hat{w}^{-1} \mu}^{J_{c}, -} 
\end{align*}

\end{theorem}
\noindent The RHS is expressed in terms of the polynomial representation $\pi$ of the Hecke algebra and partial symmetrizers  $\mathbf{1}^{\pm}_{J_c}$ with respect to the stabilizer subgroup $W_{c}$ (see Definition \ref{def:partial_symm}).  Also, $\iota : \mathbb{F}[\Lambda] \rightarrow \mathbb{F}[\Lambda]$ is given by $\iota(x^{\lambda}) = x^{-\lambda}$ and extending linearly.  In the second statement, the polynomials $p^{J, \pm}_{\hat{w}^{-1} \mu} = p^{J, \pm}_{\hat{w}^{-1} \mu}(t)$ are the $q \rightarrow \infty$ limits of Macdonald polynomials with prescribed symmetry, as introduced in \cite{BDF} and studied further in \cite{Mar, B}.  In \cite{BBBG19}, the $q \rightarrow \infty$ limits $p^{J, -}_\mu$ have been related to parahoric Whittaker functions, which is consistent with the relationship between $\overline{E}^-_y$ and metaplectic spherical Whittaker functions \cite[Theorem 10.5(2)]{SSV2}, and our Theorem \ref{cor:Whitt_typeA}. Note that $p^{J, \pm}_{\hat{w}^{-1} \mu}$ interpolates between nonsymmetric Macdonald polynomials ($J = \emptyset$) and (anti-)symmetric Macdonald polynomials ($J = [1,r]$) as one varies $J$.  See Section \ref{sec:lim_ptl} for a generalization to arbitrary types.  Following \cite{BBBG21} in the Iwahori-metaplectic context, we refer to the surprising connection between these quasi-polynomials and classical Macdonald polynomials in the $q \rightarrow \infty$ limit, as illustrated in Theorem \ref{cor:qp_main_GL}, as duality. Note that this duality does not appear to extend directly to the $q$-level, see Example \ref{ex:qlevel}.  

Note that the quasi-polynomial representation and associated quasi-polynomials are constructed in \cite{SSV2} for $c$ in the larger set $\overline{C^+} = \{ (y_1, \dots, y_r) \in E : 0 \leq y_i - y_{i+1} \leq 1 \text{ for } 1 \leq i \leq r-1 \text{ and } y_1 - y_r \leq 1 \}$. The techniques in this paper require $c \in C_\Lambda^0$, the key issue being that we need the decomposition \eqref{qp_gamma}. We are thus only able to obtain results for $y \in E_{\Lambda}^0$, the union of $\widetilde{\mathcal{O}}_c$ for $c \in C_\Lambda^0$. For $GL_r$ this does not in fact impose a restriction as $E = E_\Lambda^0$. For other types we suspect that our techniques can be extended to handle the case $c \in \overline{C^+}$ using parabolic subgroups of the affine Weyl group $\widetilde{W}$ and associated partial symmetrizers in $\widetilde{H}_\Lambda$, but this is beyond the  scope of the current paper. 
 
 The results of this paper rely on computing PBW-type expansions of certain elements in the Hecke algebra (see Theorem \ref{thm:main_thm_HA} and Corollaries \ref{cor:PBW1pm}, \ref{cor:1fT}).  We then use the quasi-polynomial and metaplectic representations (resp.) of Hecke algebras to obtain decomposition formulas for the classes of quasi-polynomials described above, as well as for metaplectic spherical Whittaker functions (resp.) (see Theorem \ref{thm:qp_1pfT} and \ref{thm:qp_1mfT}, as well Theorem \ref{cor:qp_main} for quasi-polynomials, and Theorem \ref{thm:Whitt_arb} for metaplectic Whittaker functions in arbitrary type).  Specializing to the $GL_{r}$ case, we recover results of \cite{BBBG21}, as well as establish the authors' ``parahoric-metaplectic" duality conjecture by proving Theorem \ref{cor:Whitt_typeA}.  Finally, in Section \ref{sec:symm_qps} we use the methods of this paper to show that (anti-)symmetric quasi-polynomials, as defined in Section 6.6 of \cite{SSV2}, are in bijection with partially (anti-)symmetric polynomials (Corollary \ref{cor:bijection}).
 
 %%%%%%%%%%%%%%%%%%%%%%%%

 %%%%%%%%%%%%%%%%%%%
{\it Acknowledgements:} The author would like to thank A. Aggarwal, A. Borodin, B. Brubaker, P. Etingof, H. Gustafsson, S. Sahi, and J.V. Stokman for many useful discussions about this work.
%%%%%%%%%%%%%%%%%%%

%%%%%%%%%%%%%%%%%%%%%%%%%%%%%%%%%%%%%%%%
\section{Hecke algebras}\label{sec:prelim}
%%%%%%%%%%%%%%%%%%%%%%%%%%%%%%%%%%%%%%%%
\subsection{Root systems and Weyl groups}

We set up some preliminaries that will be used throughout the paper (see e.g., \cite{Lu, Hu} for more background on these topics).

Let $E$ be an Euclidean space with scalar product $\langle\cdot,\cdot\rangle$ and corresponding
norm $\|\cdot\|$. Let $E^*$ be its linear dual. We turn $E^*$ into an Euclidean space by transporting the scalar product of $E$ through the linear isomorphism $E\overset{\sim}{\longrightarrow}
E^*$, $y\mapsto \langle y,\cdot\rangle$. The resulting scalar product and norm on $E^*$ are denoted by $\langle\cdot,\cdot\rangle$ and $\|\cdot\|$ again. 

Let $\Phi\subset E^*$ be a reduced irreducible root system in $\textup{span}_{\mathbb{R}}\{\alpha\, | \, \alpha\in\Phi\}$. We normalize $\Phi$ in such a way that long roots have squared length equal to $2/m^2$ for some $m\in\mathbb{Z}_{>0}$. Set
$E_{\textup{co}}:=\cap_{\alpha\in\Phi}\textup{Ker}(\alpha)$, 
and write $E^\prime\subseteq E$ for the orthogonal complement of $E_{\textup{co}}$ in $E$,
so that $E=E^\prime\oplus E_{\textup{co}}$ (orthogonal direct sum). Write $\textup{pr}_{E^\prime}$ and $\textup{pr}_{E_{\textup{co}}}$ for the corresponding projections onto
$E^\prime$ and $E_{\textup{co}}$, respectively.

For $\alpha\in\Phi$ write $s_\alpha\in\textup{GL}(E^*)$ for the reflection
\begin{equation*}
s_\alpha(\xi):=\xi-\xi(\alpha^\vee)\alpha\qquad (\xi\in E^*),
\end{equation*}
where $\alpha^\vee\in E$ is the unique vector such that 
\begin{equation*}
\langle y,\alpha^\vee\rangle=
2\alpha(y)/\|\alpha\|^2
\end{equation*}
 for all $y\in E$.
The finite Weyl group $W$ of $\Phi$ is the subgroup of $\textup{GL}(E^*)$ generated by $s_\alpha$ ($\alpha\in\Phi$). It is a Coxeter group, with Coxeter generators the simple reflections $s_i:=s_{\alpha_i}$ ($1\leq i\leq r$).  The root system $\Phi\subset E^*$ is $W$-invariant.  We write $w_0\in W$ for the longest Weyl group element, and $\varphi\in\Phi^+$ for the highest root.  Note that $E^\prime$ is $W$-invariant, and $W$ acts trivially on $E_{\textup{co}}$.

Write $\Phi^\vee=
\{\alpha^\vee\}_{\alpha\in\Phi}\subset E$ for the associated coroot system.  Note that $\Phi^\vee\subset E^\prime$,
and $E^\prime=\textup{span}_{\mathbb{R}}\{\alpha^\vee \,\, | \,\, \alpha\in\Phi\}$.
The coroot lattice and co-weight lattice of $\Phi$ are
\begin{equation*}
Q^\vee:=\mathbb{Z}\Phi^\vee,\qquad P^\vee:=\{\lambda\in E^\prime \,\, | \,\, \alpha(\lambda)\in\mathbb{Z}\quad \forall\, \alpha\in\Phi\}.
\end{equation*}
They are full sub-lattices of $E^\prime$, and $Q^\vee\subseteq P^\vee$.

Consider the action of $W$ on $E$ with $s_\alpha$ ($\alpha\in\Phi$) acting as the orthogonal reflection in the hyperplane $\alpha^{-1}(0)$,
\begin{equation*}
s_\alpha(y):=y-\alpha(y)\alpha^\vee\qquad (y\in E).
\end{equation*}
Then $(w \xi)(y)=\xi(w^{-1}y)$ for $v\in W$, $\xi\in E^*$ and $y\in E$, hence 
\begin{equation*}
w(\alpha^\vee)=(w\alpha)^\vee\qquad\quad (w\in W,\,\alpha\in\Phi).
\end{equation*}
In particular, $\Phi^\vee\subset E$ is $W$-invariant, and hence so are the lattices $Q^\vee$ and $P^\vee$. 

Fix a set $\Delta:=\{\alpha_1,\ldots,\alpha_r\}$ of simple roots and write $\Phi=\Phi^+\cup\Phi^-$ for the resulting natural division of the root system in positive and negative roots. Let $\Delta^\vee:=\{\alpha_1^\vee,\ldots,\alpha_r^\vee\}$ be the associated set of simple coroots for $\Phi^\vee$ and $\Phi^{\vee,\pm}$ the associated sets of positive and negative coroots.
We have $P^\vee=\bigoplus_{i=1}^r\mathbb{Z}\varpi_i^\vee$ with
$\varpi_i^\vee\in E^\prime$ ($1\leq i\leq r$) the fundamental weights with respect to $\Delta$, characterized by $\alpha_i(\varpi_j^\vee)=\delta_{i,j}$ ($1\leq i,j\leq r$). Set 
$P^{\vee,\pm}:=\pm\sum_{i=1}^r\mathbb{Z}_{\geq 0}\varpi_i^\vee$ for the cones of dominant and anti-dominant co-weights in $P^\vee$, respectively. 

\begin{definition}\label{latticeconddef}
	Denote by $\mathcal{L}$ the set of finitely generated abelian subgroups $\Lambda\subset E$ satisfying
	\begin{equation}\label{latticecond}
		Q^\vee\subseteq\Lambda\quad \&\quad \alpha(\Lambda)\subseteq\mathbb{Z}\quad \forall\, \alpha\in\Phi.
	\end{equation}
\end{definition}

Note that if $\Lambda\in\mathcal{L}$ then $\Lambda\cap E^\prime,\textup{pr}_{E^\prime}(\Lambda)\in\mathcal{L}$ and 
\begin{equation*}
	Q^\vee\subseteq \Lambda\cap E^\prime\subseteq\textup{pr}_{E^\prime}(\Lambda)\subseteq P^\vee.
\end{equation*}

A key example is the following:
\begin{example}\label{lambda_examples}$\mathbf{GL}_{r}$: Let $E = \mathbb{R}^{r}$ and set $\alpha_i^{\vee} = e_{i+1} - e_i$ for $1 \leq i \leq r-1$, where $\{e_i\}_{1 \leq i \leq r}$ are the canonical basis vectors  (note the rank is $r-1$). Set $\Lambda = \mathbb{Z}^{r} \subset E$, then $\Lambda \in \mathcal{L}$ and we have
		\begin{align*}
			E' &= \{(v_1, \dotsc, v_{r}) \in \mathbb{R}^{r}: v_1 + \dotsb + v_{r} = 0\} \\
			Q^{\vee} &= \Lambda \cap E' \\
			P^{\vee} &= \{(v_1, \dotsc, v_{r}) \in E': v_{i+1} - v_i \in \mathbb{Z} \text{ for $1 \leq i \leq r-1$}\} \\
			E_{\textup{co}} &= \{(c,c, \dots, c) \in \mathbb{R}^{r} \}.
		\end{align*}
\end{example}

For $\Lambda \in \mathcal{L}$, the $\Lambda$-extended affine Weyl group is the semi-direct product group $\widetilde{W}_{\Lambda}:=W\ltimes \Lambda$. When $\Lambda = Q^\vee$ this is the (non-extended) affine Weyl group $\widetilde{W} = \widetilde{W}_{Q^\vee}$.  We extend the $W$-action on $E$ to a faithful $\widetilde{W}_\Lambda$-action on $E$ by affine linear transformations, with $\lambda \in \Lambda$ acting by
\begin{equation*}
\tau(\lambda)y:=y+\lambda \qquad (y\in E).
\end{equation*}
We will regularly identify $\widetilde{W}_{\Lambda}$ with its realization as a subgroup of the group of affine linear automorphisms of $E$.

The action of $\widetilde{W}$ on $E$ preserves  
\begin{equation}\label{Eareg}
	E^{a,\textup{reg}}:=\{y\in E \,\, | \,\, 
	\alpha(y)\notin \mathbb{Z} \quad \forall \alpha \in\Phi \}.
\end{equation}
The connected components of $E^{a,\textup{reg}}$ are called alcoves. The resulting $\widetilde W$-action on the set of alcoves is simply transitive. The alcove
\begin{align}\label{Cplus}
C^+:&=\{ y\in E \,\, | \,\, 0<\alpha(y)< 1\quad \forall\, \alpha\in\Phi^+ \} \\
&= \{y \in E  \,\, | \,\, 0 < \alpha_i (y) \text{ for } 1 \leq i \leq r \text{ and } \varphi(y) < 1 \}
\end{align}
is the fundamental alcove in $E$. We define
\begin{equation}\label{eq:omega_def}
	\Omega_{\Lambda} := \{w \in \widetilde{W}_\Lambda: w(C^+) \subset C^+\}.
\end{equation}
Note that $\Omega_{\Lambda} = \{1\}$ for $\Lambda = Q^\vee$.

For $y\in E$ we write 
\begin{equation}\label{O} 
	\widetilde{\mathcal{O}}_y:=\widetilde{W}y = \{ \lambda + wy : \lambda \in Q^{\vee}, w \in W \}
\end{equation}
for the $\widetilde{W}$-orbit of $y$ in $E$. We thus have
\begin{equation}\label{WorbitE}
	E^{a,\textup{reg}}=\bigsqcup_{c \in C^+}\widetilde{\mathcal{O}}_{c},\qquad\quad E=\bigsqcup_{c\in\overline{C^+}}\widetilde{\mathcal{O}}_{c}
\end{equation}
(these are disjoint unions).

Likewise, for $y \in E$ and $\Lambda \in \mathcal{L}$ we define
\begin{equation*}
\widetilde{\mathcal{O}}_{\Lambda, y} := \widetilde{W}_{\Lambda,y} = \{ \lambda + wy : \lambda \in \Lambda, w \in W \}.
\end{equation*}
Note that $\widetilde{\mathcal{O}}_y = \widetilde{\mathcal{O}}_{Q^\vee, y} \subseteq \widetilde{\mathcal{O}}_{\Lambda, y}$. Note that $\widetilde{W}_{\Lambda, y}$ contains the finite Weyl group stabilizer of $y$, 
\begin{equation*}
W_y = \{w \in W \,\, | \,\, w(y) = y \}.
\end{equation*}

Throughout this paper, we will be interested in situations where $\widetilde{W}_{\Lambda, c} = W_{c}$ for $c \in \overline{C^+}$.  We use the following notation.
 \begin{definition}\label{def:key_ccondit}
For $\Lambda \in \mathcal{L}$, we define 
\begin{equation}
C^0_\Lambda := \{c  \in \overline{C^+}: \widetilde{W}_{\Lambda, c} = W_c\}.
\end{equation}
\end{definition}

Note that $\varphi(c) < 1$ is a necessary condition for $c \in C^{0}_{\Lambda}$, since otherwise $s_0 := \tau(\varphi)s_\varphi \in \widetilde{W}_{Q^\vee, c} \setminus W_{c}$. The following Lemma provides a complete characterization.

\begin{lemma}
Let $c \in \overline{C^+}$. Then $c \in C^{0}_\Lambda$ if and only if $\varphi(c) < 1$ and
\begin{equation*}
	\Omega_{\Lambda, c} := \{\omega \in \Omega_{\Lambda}: \omega(c) = c\} = \{1\}.
\end{equation*}

\end{lemma}
\begin{proof}
It is well-known that $\widetilde{W}_{Q^\vee, c}$ is parabolic,  and is equal to $W_c$ if and only if $\varphi(c) < 1$. The result for general $\Lambda$ then follows from $\widetilde{W}_{\Lambda, c} = \Omega_{\Lambda, c} \rtimes \widetilde{W}_{Q^\vee, c}$
\end{proof}

Note that $\varphi(c) < 1$ is thus a necessary and sufficient condition for $\Lambda = Q^\vee$ (as $\Omega_{Q^\vee} = \{1\}$) and in the $GL_{r}$ setting (as $\Omega_{\Lambda, c} = \{1\}$ for all $c \in \overline{C}^+$, see \cite[Lemma 7.31]{SSV2}).

For $c \in C^0_\Lambda$, we define
\begin{equation}
	W^{c} := \{w \in W: w \textrm{ is a minimal length coset representative of } W / W_{ c }\}.
\end{equation}
 It follows from the definitions that $\mathbb{F}[\widetilde{\mathcal{O}}_{\Lambda, c }]$ is a free $\mathbb{F}[\Lambda]$-module, with basis given by $\{x^{w c}: w \in W^{c}\}$. Note this is not the case when $c \notin C_{\Lambda}^0$, which is why we need this technical condition throughout the paper.

%%%%%%%%%%%%%%%%%%%%
\subsection{Length function and parabolic subgroups}\label{cosetsection}
%%%%%%%%%%%%%%%%%%%%

We recall some properties of the length function
\begin{equation}\label{length}
	\ell(w):=\textup{Card}\bigl(\Pi(w)\bigr),\qquad \Pi(w):=\Phi^+\cap w^{-1}\Phi^-
\end{equation}
that we will frequently use.
%%%%%%%%%%%%%%%%%
\begin{lemma}\label{lengthadd}
	Let $1\leq j\leq r$ and $w\in W$. 
	\begin{enumerate}  
		\item $|\ell(s_jw)-\ell(w)|=1$.
		\item We have
		\[
		\ell(s_jw)=\ell(w)+1 \quad\Leftrightarrow \quad w^{-1}\alpha_j\in \Phi^+,
		\]
		and then $\Pi(s_jw)=\{w^{-1}\alpha_j\}\cup \Pi(w)$
		\textup{(}disjoint union\textup{)}.
	\end{enumerate}
\end{lemma}
%%%%%%%%%%%%%%%%%%
\begin{proof}
	See, e.g., \cite[(2.2.8)]{Ma}.
\end{proof}
%%%%%%%%%%%%%%%%%%
The length $\ell(w)$ of $w\in W$ is the smallest nonnegative integer $\ell$ for which there exists an expression $w=s_{j_1}\cdots s_{j_\ell}$ of $w$ as product of $\ell$ simple reflections 
($1 \leq j_i\leq r$).
Such shortest length expressions are called reduced expressions. If
$w=s_{j_1}\cdots s_{j_\ell}$ is a reduced expression then 
\begin{equation}\label{rootsdescription}
	\Pi(w)=
	\{b_1,\ldots,b_\ell\}\,\,\, \hbox{ with }\, b_m:=s_{j_\ell}\cdots s_{j_{m+1}}(\alpha_{j_m})\,\,\, (1\leq m<\ell),\quad b_\ell:=\alpha_{j_\ell}.
\end{equation}

Fix a subset $J\subseteq [1,r]$. The associated parabolic subgroup $W_J\subseteq W$ is the subgroup of $W$ generated by 
$s_j$ ($j\in J$). Write $W^{J}$ for the set of elements $w\in W$ such that 
\begin{equation*}
\ell(ww^\prime)=\ell(w)+\ell(w^\prime)\qquad \forall\, w^\prime\in W_{J}.
\end{equation*}
Write $\Phi_{J}:=\Phi\cap\textup{span}_{\mathbb{Z}}\{\alpha_j\}_{j\in J}$ and $\Phi_J^{\pm}:=\Phi^{\pm}\cap\Phi_J$, then 
\begin{equation}\label{Phipos}
	W^{J}=\{w\in W \,\, | \,\, w\Phi_J^+\subseteq\Phi^+ \}.
\end{equation}
%%%%%%%%%%%%%%%%%%%%%%%%%%%%%%%%%%%%%%%%
\begin{lemma}\label{cosetcomb}
	Fix a subset $J\subseteq [1,r]$.
	\begin{enumerate}
		\item $W^J$ is a complete
		set of representatives of $W/W_J$ \textup{(}the elements in $W^J$ are called the minimal coset representatives of $W/W_J$\textup{)}.
		\item Let $w\in W^J$ and $1 \leq j\leq r$. Then 
		\[
		s_jw\not\in W^J\,\,\,\Leftrightarrow\,\,\, s_jwW_J=wW_J,
		\]
		and then $s_jw=ws_{j^\prime}$ for some $j^\prime\in J$ \textup{(}in particular, 
		$\ell(s_jw)=\ell(w)+1$\textup{)}.
	\end{enumerate}
\end{lemma}
%%%%%%%%%%%%%%%%%%%%%%%%%%%%%%%%%%%%%%%%
\begin{proof}
	(1) See, e.g., \cite{B}.\\
	(2) This follows from \cite[Lem. 3.1 \& Lem. 3.2]{D}. 
\end{proof}
%%%%%%%%%%%%%%%%%%%%%%%%%%%%%%%%%%%%%%
The stabilizer subgroup $W_{c}$ of $c\in\overline{C}_+$ is parabolic, 
\[
W_{c}=W_{J_c}
\]
with $J_c \subseteq [1,r]$ given by
\begin{equation}\label{Jc}
	J_c :=\{j\in [1,r] \,\, | \,\, s_jc=c\}.
\end{equation}
%%%%%%%%%%%%%%%%%%%%%%%%%%%%%%%%%%%%%%%%%%

\begin{definition}\label{def:wc_decomp}
Let $w \in W$ and $c \in \overline{C}_{+}$.  We will write $w^{c}$, $w_c$ (resp.) for the (unique) elements of $W^{c}$, $W_{c}$ (resp.) such that $w = w^{c} w_{c}$.
\end{definition}

\begin{lemma}\label{lem:w0_conj}
Let $w \in W$ and $c \in \overline{C}_{+}$.  Then
\begin{align*}
(w_0 w w_0)^{-w_0 c} &= w_0 w^{c} w_0\\
(w_0 w w_0)_{-w_0 c} &= w_0 w_{c} w_0.
\end{align*}
\end{lemma}
\begin{proof}
	Follows easily from the definitions.
\end{proof}

\subsection{Affine Hecke algebras}
In this section, we recall basics about Hecke algebras, the polynomial representation, and related operators.  We also introduce coefficient functions that will be central to this paper, and establish some auxiliary results.

We start by recalling the standard divided-difference operator which appears in the commutation relation defining the Hecke algebra.

\begin{definition}\label{def:nabla} 
	For $\Lambda \in \mathcal{L}$, let $\nabla_j: \mathbb{F}[\Lambda] \rightarrow \mathbb{F}[\Lambda]$ for $1 \leq j \leq r$ be defined by
	\begin{equation*}
		\nabla_j(f) = \frac{s_j f - f}{x^{\alpha_j^{\vee}} - 1}
	\end{equation*}
	for $f \in \mathbb{F}[\Lambda]$.
\end{definition}

\begin{definition}\label{def:tparams}
Fix a $W$-invariant map $\mathbf{t}: \alpha \mapsto t_{\alpha} \in \mathbb{F}$ and set $t_j := t_{\alpha_j}, t(w) := \prod t_{j_i}$ when $w = s_{j_1} \cdots s_{j_l} \in W$.
\end{definition}

\begin{definition}\label{def:AHA}
For $\Lambda \in \mathcal{L}$, the $\Lambda$-extended affine Hecke algebra $\widetilde{\mathcal{H}}_{\Lambda}$ is the unital associative $\mathbb{F}$-algebra generated by $T_1, \dots, T_r$ and $x^{\lambda}$ for $\lambda \in \Lambda$, subject to the following relations
\begin{enumerate}
\item The braid relations
\begin{equation}\label{braid}
T_{j} T_{j'} T_{j} \dots = T_{j'} T_{j} T_{j'} \dots (\text{with } m_{jj'} \text{ factors on each side})
\end{equation}
for $1 \leq j \neq j' \leq r$, where $m_{jj'}$ is the order of $s_{j}s_{j'}$ in $W$,
\item The quadratic relations
\begin{equation}\label{quadratic}
(T_j - t_j)(T_j + t_{j}^{-1}) = 0
\end{equation}
for $1 \leq j \leq r$,
\item $x^{\lambda} x^{\mu} = x^{\lambda + \mu}$ for $\lambda, \mu \in \Lambda$ and $x^{0} = 1$,
\item The commutation relations
\begin{equation}\label{commutation}
T_{j} x^{\lambda} - x^{s_j \lambda} T_{j} = (t_j - t_{j}^{-1}) \nabla_i(x^\lambda)
\end{equation}
for $1 \leq j \leq r$ and $\lambda \in \Lambda$.
\end{enumerate}
\end{definition}
Note also that $T_{j} \in \widetilde{\mathcal{H}}_{\Lambda}^{\times}$ with inverse 
\begin{equation}\label{eq:Tj_inv}
T_{j}^{-1} = T_{j} - t_j + t_{j}^{-1}, 
\end{equation}
which follows from the quadratic relation.

\begin{remark}\label{rem:met_HA}
Recall from \cite{SSV, SSV2} that for a metaplectic parameter $m \in \mathbb{Z}_{\geq 1}$, one may define the metaplectic Hecke algebra $\widetilde{\mathcal{H}}_{\Lambda^m}$.  It is the $\Lambda^m$-extended affine Hecke algebra associated to $\Phi^m$ and has generators $T_1, \dots, T_r$ and $x^{\lambda}$ for $\lambda \in \Lambda^m$.  This will be discussed further in Section \ref{sec:metreps}.
\end{remark}

For $w\in W$ with reduced expression $w=s_{j_1}\cdots s_{j_\ell}$
($1\leq j_i\leq \ell$) we write 
\begin{equation*}
	T_w:=T_{j_1}\cdots T_{j_\ell}\in\widetilde{\mathcal{H}}_{\Lambda},
\end{equation*}
which is independent of the choice of reduced expression, due to the braid relations \eqref{braid}.

It follows from the quadratic relation that we then have
\begin{equation}\label{TT}
T_i T_w = T_{s_iw} + \delta_{l(s_iw) < l(w)} (t_i - t_i^{-1})T_w.
\end{equation}
If $u, v \in W$ with $l(uv) = l(u) + l(v)$, then we have
\begin{equation}\label{Tw_prod}
	\begin{split}
		T_{uv} &= T_u T_v \\
		T_{(uv)^{-1}}^{-1} &= T_{u^{-1}}^{-1} T_{v^{-1}}^{-1}.
	\end{split}
\end{equation}

From \eqref{commutation}, we have
\begin{equation}\label{Tf_commut}
T_j f = s_j f T_j + (t_j - t_j^{-1}) \Big( \frac{s_j f - f }{x^{\alpha_j^{\vee}} - 1} \Big)
\end{equation}
for $1 \leq j \leq r$ and $f \in \mathbb{F}[\Lambda]$.  Similarly, we also have
\begin{equation}\label{Tinvf_commut}
T_j^{-1} f = s_j f \big(T_j - (t_j - t_j^{-1})\big) + (t_j - t_j^{-1}) \Big( \frac{s_j f - f}{1-x^{-\alpha_j^{\vee}}}\Big).
\end{equation}

Let $w_0 \in W$ be the longest Weyl group element.  We will need the following results about $w_0$.
\begin{lemma}\label{lem:Tw0}
	For any $w \in W$, we have
	\begin{align*}
		T_{w w_0} &= T_{w^{-1}}^{-1} T_{w_0} \\
		T_{w_0 w} &= T_{w_0} T_{w^{-1}}^{-1}.
	\end{align*}
\end{lemma}
\begin{proof}
Note that $l(w w_0) = l(w_0 w) = l(w_0) - l(w)$.  Then the first equation follows from Eq.  \eqref{Tw_prod} with $u=w^{-1}$, $v=w w_0$, the second with $u=w_0w$ and $v=w^{-1}$.
\end{proof}

\begin{lemma}\label{lem:lensplit}
Let $w \in W^{c}$ and $u \in W_{c}$.  Then $l(wu) = l(w) + l(u)$.
\end{lemma}

\begin{lemma}\label{lem:Tw0wu}
	Let $w \in W^{c}$ and $u \in W_{c}$.  We have
	\begin{align*}
		T_{(w_0 w u)^{-1}}^{-1} & = T_{w_0}^{-1} T_{w} T_{u} \\
		T_{w_0 w u} &= T_{w_0} T_{w^{-1}}^{-1} T_{u^{-1}}^{-1}.
	\end{align*}
\end{lemma}
\begin{proof}
	Follows from Lemma \ref{lem:Tw0} and Lemma \ref{lem:lensplit}.
\end{proof}

\begin{definition}\label{def:iota} 
Let $\iota: \mathbb{F}[E]\rightarrow \mathbb{F}[E]$ defined by $\iota(x^{\mu}) = x^{-\mu}$ and extending linearly. Note that $\iota$ preserves $\mathbb{F}[\Lambda]$.
\end{definition}

One can easily verify the following identities: 

\begin{proposition}\label{w0_facts}
Let $1 \leq i \leq r$.  Then
\begin{enumerate}
\item $w_0^{2} = 1$
\item $\iota w = w \iota$ for $w \in W$
\item $w_0 \alpha_i^{\vee} = -\alpha_{w_0(i)}^{\vee}$
\item $s_i w_0 = w_0 s_{w_0(i)}$.
\end{enumerate}
\end{proposition}

\begin{lemma}\label{w0sinabla}
Let $f \in \mathbb{F}[\Lambda]$, $1 \leq i, j \leq r$, and $w_0(i) = j$.  We have
\begin{equation*}
w_0 s_i \nabla_i f = -\nabla_{j}(w_0 f).
\end{equation*}
\end{lemma}
\begin{proof}
We compute
\begin{equation*}
w_0 s_i \nabla_i f = \frac{w_0 f - s_j w_0 f}{x^{\alpha_j^{\vee}} - 1} = -\nabla_{j}(w_0 f)
\end{equation*}
using Proposition \ref{w0_facts}.
\end{proof}

\begin{theorem}\label{thm:pol_rep}
For $\Lambda \in \mathcal{L}$, the polynomial representation $\pi$ of $\widetilde{\mathcal{H}}_{\Lambda}$ on $\mathbb{F}[\Lambda]$ is given by the following formulas:
\begin{equation*}
\pi(T_j) x^{\lambda} = t_j x^{s_j \lambda} + (t_j - t_j^{-1}) \nabla_j (x^{\lambda})
\end{equation*}
for $\lambda \in \Lambda$, $1 \leq j \leq r$, and $\pi(x^{\mu}) x^{\lambda} = x^{\lambda + \mu}$ for $\mu \in \Lambda$.

\end{theorem}

\begin{remark}
The operators $\pi(T_j)$ for $1 \leq j \leq r$ are called Demazure-Lusztig operators.
\end{remark}

\begin{definition}\label{def:poly_rep_v}
We write $\pi_{v}$ for the representation $\pi :  \widetilde{\mathcal{H}}_{\Lambda} \rightarrow \text{End}(\mathbb{F}[\Lambda])$ of Theorem \ref{thm:pol_rep} with all Hecke parameters $t_i = v^{1/2}$ for all $1 \leq i \leq r$.
\end{definition}

Using \eqref{Tinvf_commut} and Theorem \ref{thm:pol_rep}, we have the following formula for $\pi(T_j^{-1})f$, with $f \in \mathbb{F}[\Lambda]$ and $1 \leq j \leq r$:
\begin{equation}\label{pi_Tinvf}
\pi(T_j^{-1})f = t_j^{-1} s_j f +  (t_j - t_j^{-1}) \Big( \frac{s_j f - f}{1-x^{-\alpha_j^{\vee}}}\Big).
\end{equation}

\begin{lemma}\label{lem:sinablai}
	For $1 \leq i \leq r$, we have
	\begin{equation*}
		s_i\nabla_i = \nabla_i + s_i - 1.
	\end{equation*}
\end{lemma}

\begin{proof}
Let $1 \leq i \leq r$ and $f \in \mathbb{F}[\Lambda]$.  We compute directly using Definition \ref{def:nabla}
\begin{equation*}
s_i \nabla_i(f) = \frac{x^{\alpha_i^{\vee}}  (s_i f -  f) }{x^{\alpha_i^{\vee}} -1 } = \frac{(s_i f - f) + (x^{\alpha_i^{\vee}} - 1)s_i f - (x^{\alpha_i^{\vee}}-1)f}{x^{\alpha_i^{\vee}} -1} = (\nabla_i + s_i - 1)f.
\qedhere
\end{equation*}	
\end{proof} 

\begin{lemma}\label{lem:w0Ti}
Let $1 \leq i \leq r$.  We have
\begin{align*}
\pi(T_i^{\pm 1}) \iota w_0 &= \iota w_0 \pi(T_{w_0(i)}^{\pm 1}) \\
\pi(T_i) w_0 &= w_0\bigl(-\pi(T_{w_0(i)}^{-1}) + (t_i + t_i^{-1})s_{w_0(i)}\bigr).
\end{align*}
\end{lemma}
\begin{proof}
Let $\lambda \in \Lambda$ and $1 \leq i \leq r$.  

Consider the first statement.  We have
\begin{align*}
\pi(T_i) \iota w_0 (x^{\lambda}) &= \pi(T_i) x^{-w_0\lambda} = t_i x^{-s_i w_0 \lambda} + (t_i - t_i^{-1})\nabla_i(x^{-w_0(\lambda)}) \\
\pi(T_i^{-1}) \iota w_0 (x^{\lambda}) &= (\pi(T_i) - (t_i - t_i^{-1})) x^{-w_0 \lambda} \\
&= t_i x^{-s_i w_0 \lambda} + (t_i - t_i^{-1})\nabla_i(x^{-w_0(\lambda)})  - (t_i - t_i^{-1})x^{-w_0 \lambda}.
\end{align*}
Using Proposition \ref{w0_facts} and Definition \ref{def:nabla} gives the first assertion.
The argument for the second statement is similar.
\end{proof}

We write $\pi(\widetilde{T_i})$ for the operators with Hecke parameters $t_i$ replaced by $t_i^{-1}$.  We record the following relation.
\begin{proposition}\label{prop:inv_params}
Let $1 \leq i \leq r$.  Then, as operators on $\mathbb{F}[\Lambda]$, we have
\begin{equation*}
\pi(\widetilde{T_i}^{-1}) = \iota \pi(T_i) \iota.
\end{equation*}
\end{proposition}
\begin{proof}
Follows from \eqref{pi_Tinvf}.
\end{proof}

Let $h \in \widetilde{\mathcal{H}}_{\Lambda}$ be an arbitrary element.  By the PBW Theorem for affine Hecke algebras, there is a unique decomposition 
\begin{equation*}
h = \sum_{w \in W} f_{w} T_{w},
\end{equation*}
for $f_{w} \in \mathbb{F}[\Lambda]$.  In other words, $\widetilde{\mathcal{H}}_{\Lambda}$ is a free $\mathbb{F}[\Lambda]$-module, with basis given by $\{T_{w} : w \in W \}$.

\begin{definition}\label{def:gamma_coeffs_Hecke}
Define the collection of $\mathbb{F}[\Lambda]$-linear maps $\gamma_{w} : \widetilde{\mathcal{H}}_{\Lambda} \rightarrow \mathbb{F}[\Lambda]$ for $w \in W$ by 
\begin{equation*}
\gamma_{w}(  \sum_{w' \in W} f_{w'} T_{w'}) = f_{w},
\end{equation*} 
where $f_{w'} \in \mathbb{F}[\Lambda]$ for $w' \in W$.
\end{definition}

In Section \ref{sec:AHA_results}, we will compute the coefficients $\gamma_{w}(h)$ for specific elements $h \in \widetilde{\mathcal{H}}_{\Lambda}$ and arbitrary $w \in W$.   The following lemmas will be useful.

\begin{lemma}\label{gamma_Th}
	Let $1 \leq i \leq r$, $h \in \widetilde{\mathcal{H}}_{\Lambda}$ and $w \in W$.  Then
	\begin{equation*}
		\gamma_{w}(T_i h) = (t_i - t_i^{-1}) \nabla_{i}(\gamma_{w}(h)) + s_i \gamma_{s_i w}(h) + \delta_{l(s_i w) < l(w)} (t_i - t_i^{-1}) s_i \gamma_{w}(h).
	\end{equation*}
\end{lemma}

\begin{proof}
	Setting $f_{\hat w} = \gamma_{\hat w}(h) \in \mathbb{F}[\Lambda]$, and using \eqref{Tf_commut} and \eqref{TT}, we have
	\begin{align*}
		T_i h &= T_i \sum_{\hat{w} \in W} f_{\hat w} T_{\hat w} \\
		& = \sum_{\hat w \in W} \big(s_i f_{\hat w} T_{s_i \hat w} + (t_i - t_i^{-1}) \nabla_i(f_{\hat w}) T_{\hat w} \big) + \sum_{\hat w : l(s_i \hat{w}) < l(\hat{w})} (t_i - t_i^{-1}) (s_i f_{\hat w}) T_{\hat w}.
	\end{align*}
	Taking the coefficient on $T_{w}$ gives the desired result. 
\end{proof}

\begin{lemma}\label{lem:T_triang}
Let $f \in \mathbb{F}[\Lambda]$ and $w \in W$.  Then if $l(w') \leq l(w)$, we have
\begin{equation*}
\gamma_{w}(T_{w'}^{\pm 1} f) = \delta_{\{ w = w' \}} wf.
\end{equation*}
\end{lemma}

\begin{proof}

We will prove the statement by induction on $l(w')$.  The base case $l(w') = 0$ follows directly.  Now assume the statement holds for $l(w') = j$, $w' \in W$, and suppose $l(s_i w') > l(w')$ for some $1 \leq i \leq r$.  Suppose $w \in W$ with $l(w) \geq l(s_i w')$.  By Lemma \ref{gamma_Th}, we have
\begin{multline*}
\gamma_{w}(T_{s_i w'} f) = \gamma_{ w}(T_{i} T_{w'} f) \\
= (t_i - t_i^{-1}) \nabla_{i}(\gamma_{w}(T_{w'} f)) + s_i \gamma_{s_i w}(T_{w'} f) + \delta_{l(s_i w) < l(w)} (t_i - t_i^{-1}) s_i \gamma_{w}(T_{w'} f).
\end{multline*}
Applying the induction hypothesis, noting that $l(w) > l(w')$, $l(s_iw) \geq l(w')$, we obtain $\gamma_w(T_{s_iw'}f) = \delta_{\{w = s_iw'\}} s_iw'f$ as desired. Likewise, applying Lemma \ref{gamma_Th} to
\begin{equation*}
\gamma_{w}(T_{(s_i w')^{-1}}^{-1} f) = \gamma_{w}(T_{i}^{-1} T_{{w'}^{-1}}^{-1} f) = \gamma_{w}(T_{i}T_{{w'}^{-1}}^{-1} f) - (t_i - t_i^{-1}) \gamma_w(T_{{w'}^{-1}}^{-1} f),
\end{equation*}
we find $\gamma_w(T_{(s_iw')^{-1}}^{-1}f) = \delta_{\{w = (s_iw')^{-1}\}} (s_iw')^{-1} f$. The result now follows from induction.
\end{proof}

%%%%%%%%%%%%%%%%%%%%%%%%%%%%%%%%

\subsection{Partial (anti-)symmetrizers}\label{sec:partial_symm}

In this section, we recall the definition of partial (anti-)symmetrizers in $\widetilde{\mathcal{H}}_{\Lambda}$, as well as some of their main properties.  We will use these in Section \ref{sec:lim_ptl} to define partially (anti-)symmetric Macdonald polynomials. 

\begin{definition}\label{def:symm_antisymm}
Let $\mathbf{1}^{+}$ and $\mathbf{1}^{-}$ denote the symmetrizer and antisymmetrizer, respectively, in $\widetilde{\mathcal{H}}_{\Lambda}$:
\begin{align}\label{1plus1minus}
\mathbf{1}^{+} &= \sum_{w \in W} t(w) T_w\\
\mathbf{1}^{-} &= \sum_{w \in W} (-1)^{l(w)} t(w)^{-1} T_w.
\end{align}
\end{definition}

There are the following alternate expressions for $\mathbf{1}^{\pm}$:
\begin{align}\label{1plus1minus_alt}
\mathbf{1}^{+} &= t(w_0)^{2} \sum_{w \in W} t(w)^{-1} T_{w^{-1}}^{-1}\\
\mathbf{1}^{-} &= t(w_0)^{-2} \sum_{w \in W} (-1)^{l(w)} t(w) T_{w^{-1}}^{-1}.
\end{align}

Let $1 \leq i \leq r$.  We have the following identities in $\widetilde{\mathcal{H}}_{\Lambda}$: 
\begin{equation} \label{T1}
T_i \mathbf{1}^{\pm} = \pm t_i^{\pm 1} \mathbf{1}^{\pm}\\
\end{equation}
\begin{equation}  \label{T2}
\mathbf{1}^{\pm} T_i = \pm t_i^{\pm 1} \mathbf{1}^{\pm}
\end{equation}
(see \cite{Ch}).

\begin{lemma}\label{lem:antisym_w}
Let $w \in W$ and $f \in \mathbb{F}[\Lambda]$.  Then
\begin{enumerate}
\item $w(\pi(\mathbf{1}^{+}) f) = \pi(\mathbf{1}^{+}) f$
\item $\pi(\mathbf{1}^{-}) wf = (-1)^{l(w)} \pi(\mathbf{1}^{-}) f$.
\end{enumerate}
\end{lemma}
\begin{proof}
Follows from 
\begin{equation*}
s_j = t_j^{-1} \pi(T_j) - (1 - t_j^{-2}) \nabla_{j} \in End(\mathbb{F}[\Lambda])
\end{equation*}
along with \eqref{T1} and \eqref{T2}.
\end{proof}

\begin{definition}\label{def:partial_symm}
Let $J \subseteq [1,r]$ with associated parabolic subgroup $W_{J}$.  Define the $J$-partial symmetrizer and antisymmetrizer, respectively
\begin{align*}
\mathbf{1}^{+}_{J} = \mathbf{1}^{+}_{W_J} &= \sum_{w \in W_{J}} t(w) T_w\\
\mathbf{1}^{-}_{J} = \mathbf{1}^{-}_{W_J} &= \sum_{w \in W_{J}} (-1)^{l(w)} t(w)^{-1} T_w.
\end{align*}
in $\widetilde{\mathcal{H}}_{\Lambda}$.
\end{definition}

\begin{remark}
When $J = [1,r]$, we have $\mathbf{1}^{\pm}_{J} = \mathbf{1}^{\pm}$ and when $J = \emptyset$, we have $\mathbf{1}^{\pm}_{J} = 1$.
\end{remark}

\begin{definition}\label{def:Jsymm}
Let $J \subseteq [1,r]$.  We say $f \in \mathbb{F}[\Lambda]$ is $J$-partially (anti-)symmetric, respectively, if $\pi(T_j)f = t_j f$ (resp. $\pi(T_j) f = -t_j^{-1} f$) for $j \in J$.  
\end{definition}

Let $w_0(W_{J})$ denote the longest word in $W_{J}$.  Analogous to \eqref{1plus1minus_alt}, we have
\begin{proposition}\label{prop:ptl1minus}
	We have
	\begin{align*}
		\mathbf{1}^{+}_{J} &= t(w_0(W_{J}))^{2}\sum_{u \in W_J} t(u)^{-1} T_{u^{-1}}^{-1} \\
		 \mathbf{1}^{-}_{J} &= t(w_0(W_{J}))^{-2} \sum_{u \in W_J} t(u) (-1)^{l(u)} T_{u^{-1}}^{-1}.
	\end{align*}
	
\end{proposition}

\subsection{Demazure-Whittaker operators}
In later sections, we will connect our results to recent work on metaplectic Whittaker functions, so we will need some related operators.  We recall the following two definitions from \cite{BBBG21} and \cite{BBBG19}.

\begin{definition}[\cite{BBBG19, BBBG21}[Equations (30), (31) and Equation (4.3) resp.]  \label{def:DW_opers} Let $q \in \mathbb{C}$.  For $1 \leq i \leq r$, define the Demazure-Whittaker operators $\mathcal{T}_{i, q}$ acting on $\mathbb{C}[\Lambda]$ by
\begin{equation*}
\mathcal{T}_{i,q} \cdot f = \frac{x^{\alpha_i^{\vee}} - q^{-2}}{1-x^{\alpha_i^{\vee}}} s_i f + \frac{q^{-2}-1}{1-x^{\alpha_i^{\vee}}} f,
\end{equation*}
where $f \in \mathbb{C}[\Lambda]$.  For $w = s_{i_1} \cdots s_{i_r} \in W$ a reduced expression, write
\begin{equation*}
\mathcal{T}_{w,q} = \mathcal{T}_{i_1,q} \cdots \mathcal{T}_{i_r,q} \in \mathrm{End}(\mathbb{C}[\Lambda]).
\end{equation*}
\end{definition}

The (non-metaplectic) operators $\mathcal{T}_{w,q}$ will be used in Section \ref{sec:met_results} to provide formulas for metaplectic Whittaker functions.

We have the following two relations between $\pi_{q^{2}}(T_{w})$, $\pi_{q^{2}}(\mathbf{1}^{-})$, and $\mathcal{T}_{w,q }$.

\begin{proposition}\label{prop:TDW_trans}
Let $1 \leq i \leq r$ and $f \in \mathbb{C}[\Lambda]$.  We have $-q^{-1} \pi_{q^{2}}(T_{i}) f = \mathcal{T}_{i,q} f$, so for $w \in W$,
\begin{equation*}
 (-q)^{-l(w)} \pi_{q^{2}}(T_{w})f = \mathcal{T}_{w,q}f.
\end{equation*}
\end{proposition}

\begin{proof}
Let $1 \leq i \leq r$ and $\lambda \in \Lambda$.  By direct computation, we have
\begin{align*}
-q^{-1} \pi_{q^{2}}(T_{i}) x^{\lambda} &= -x^{s_i \lambda} \cdot \frac{x^{\alpha_i^{\vee}} - 1}{x^{\alpha_i^{\vee}} - 1} + (q^{-2} - 1) \frac{x^{s_i \lambda} - x^{\lambda}}{x^{\alpha_i^{\vee}} - 1} \\
&= \frac{q^{-2} - x^{\alpha_i^{\vee}} }{x^{\alpha_i^{\vee}} - 1}x^{s_i \lambda} - \frac{q^{-2} - 1}{x^{\alpha_i^{\vee}} - 1}x^{\lambda} \\
&= \mathcal{T}_{i,q} x^{\lambda}.
\end{align*}
See also Proposition 9.1 of \cite{BBBG19}.
\end{proof}

\begin{proposition}\label{prop:TDW_antisymm}
We have
\begin{equation*}
\sum_{w \in W} \mathcal{T}_{w,q} = \pi_{q^{2}}(\mathbf{1}^{-}).
\end{equation*}
\end{proposition}
\begin{proof}
Follows from Proposition \ref{prop:TDW_trans} and Definition \ref{def:symm_antisymm}.
\end{proof}

%%%%%%%%%%%%%%%%%%%%%%%%%%%%%%%%
\section{Representations and (quasi-)polynomials}

\subsection{Quasi-polynomial representations}\label{sec:qp}
In this section, we recall results from \cite{SSV2} that will be used throughout the rest of the paper.  We also define coefficient functions of quasi-polynomials, and connect to them to coefficient functions of Hecke algebra elements (see Corollary \ref{cor:gammas}).

Assume $\Lambda \in \mathcal{L}$ and, unless otherwise specified, $c \in C^0_\Lambda$.  So $\mathbb{F}[\widetilde{\mathcal{O}}_{\Lambda, c}] \subseteq \mathbb{F}[E]$ is a free $\mathbb{F}[\Lambda]$-module with basis $x^{w c}$ for $w \in W^{c}$.

Let $f \in \mathbb{F}[\widetilde{\mathcal{O}}_{\Lambda, c}] \subseteq \mathbb{F}[E]$.  We have the decomposition
\begin{equation} \label{coset_decomp}
f = \sum_{w \in W^{c}} p_{w} \cdot x^{w c},
\end{equation}
where $p_{w} \in \mathbb{F}[\Lambda]$.

\begin{definition}\label{def:gamma_coeff_qp}
Let $w \in W^{c}$.  Define coefficient functions $\gamma_{w}^{qp}: \mathbb{F}[\widetilde{\mathcal{O}}_{\Lambda, c}] \rightarrow \mathbb{F}[\Lambda]$ by
\begin{equation*}
\gamma_{w}^{qp}\big( \sum_{w' \in W^{c}} p_{w'} \cdot x^{w' c}\big) = p_{w},
\end{equation*}
where $p_{w'} \in \mathbb{F}[\Lambda]$ for $w' \in W^{c}$.  
\end{definition}

Recall the truncated divided-difference operator from \cite{SSV2}, which plays an important role in the quasi-polynomial representation.

\begin{definition}\label{def:qpnabla} 
Let $\nabla^{qp}_j: \mathbb{F}[E] \rightarrow \mathbb{F}[E]$ for $1 \leq j \leq r$ be defined by
\begin{equation*}
\nabla^{qp}_j(x^{y}) = \frac{x^{y - \lfloor \alpha_j(y) \rfloor \alpha_j^{\vee} } - x^{y} }{x^{\alpha_j^{\vee}} - 1},
\end{equation*}
where $y \in E$ (and extend by linearity to $\mathbb{F}[E]$).

\end{definition}
Note that since $\lfloor \alpha_j(y)\rangle \rfloor \in \mathbb{Z}$,
\begin{equation*}
\frac{x^{- \lfloor \alpha_j(y)\rfloor \alpha_j^{\vee}} - 1}{x^{\alpha_j^{\vee}} - 1} \in \mathbb{F}[Q^{\vee}].
\end{equation*}
Also note that $\nabla_j^{qp} |_{\mathbb{F}[P^{\vee}]} = \nabla_{j}$, and that $\nabla_j^{qp}$ preserves the subspace $\mathbb{F}[\widetilde{\mathcal{O}}_{\Lambda, c}]$ for $c \in \overline{C^{+}}$.

For $1 \leq j \leq r$, let $\chi_{j} : \mathbb{R} \rightarrow \{1, t_j \}$ be defined by
\begin{equation}\label{eq:chi_ind}
\chi_{j}(x) = \begin{cases} 1, & x \in \mathbb{R} \setminus \mathbb{Z} \\
t_j, & x \in \mathbb{Z}
\end{cases}
\end{equation}
for $x \in \mathbb{R}$.

\begin{theorem}[\cite{SSV2}]\label{thm:qp_rep}
The quasi-polynomial representation $\pi^{qp}$ of $\widetilde{\mathcal{H}}_{\Lambda}$ on $\mathbb{F}[E]$ is given by the following formulas:
\begin{equation*}
\pi^{qp}(T_{j}) x^{y} = \chi_{j}(\alpha_j(y)) \cdot x^{s_j y} + (t_j - t_j^{-1}) \nabla_{j}^{qp}(x^{y})
\end{equation*}
for $y \in E$, $1 \leq j \leq r$, and $\pi^{qp}(x^{\lambda}) x^{y} = x^{\lambda + y}$ for $\lambda \in \Lambda$ and $y \in E$.
\end{theorem}

\begin{remark} \hfill
\begin{enumerate}
\item One can directly verify that $\mathbb{F}[\widetilde{\mathcal{O}}_{\Lambda, c}]$, for $c \in \overline{C^{+}}$, is an invariant subspace of $\pi^{qp}$.

\item Let $h \in \widetilde{\mathcal{H}}_{\Lambda}$ and $f \in \mathbb{F}[\Lambda]$.  Then, comparing the formulas of Theorem \ref{thm:pol_rep} and Theorem \ref{thm:qp_rep}, we have $\pi^{qp}(h) f = \pi(h) f$.

\item The specialization of Theorem \ref{thm:qp_rep} to the $GL_r$ context gives Theorem \ref{thm:qp_rep_GL}.
\end{enumerate}
\end{remark}

We write $\pi^{qp}(\widetilde{T_i})$ for the operators with Hecke parameters $t_i$ replaced by $t_i^{-1}$.  We record the following relation, which extends Proposition \ref{prop:inv_params}.
\begin{proposition}\label{prop:inv_params_qp}
Let $1 \leq i \leq r$.  Then, as operators on $\mathbb{F}[\Lambda]$, we have
\begin{equation*}
\pi^{qp}(\widetilde{T_i}^{-1}) = \iota \pi(T_i) \iota.
\end{equation*}
\end{proposition}
\begin{proof}
Follows from Theorem \ref{thm:qp_rep} and \eqref{eq:Tj_inv}.
\end{proof}

For $c \in C_\Lambda^0$ we define root-system statistics 
\begin{equation}\label{eq:qp_h_stats}
h^{c}(w) = \prod_{\substack{\alpha \in \Pi(w)\\\alpha(c) = 0}} t_\alpha.
\end{equation}
These match the statistics $k_w(c)$ defined in \cite[(4.11)]{SSV2}.

\begin{proposition}\label{prop:ht}
We have the following:
\begin{enumerate}
	\item If $w \in W$ and $l(s_iw) > l(w)$ then
	\begin{equation*}
		h^{c}(s_iw) = t_i^{\delta_{\alpha_i(wc) = 0}}h^{c}(w).
	\end{equation*}
	\item If $\varphi(c) < 1$, $v \in W^{c}$ and $u \in W_{c}$, then we have
	\begin{equation*}
		h^{c}(vu) = t(u).
	\end{equation*}
\end{enumerate}
\end{proposition}
\begin{proof}
These are special cases of \cite[Lemma 5.6(1)]{SSV2}, and can also be proven directly with standard root system manipulations.
\end{proof}

\begin{theorem} \label{thm:XT_formula}
Let $\mu \in \Lambda$, $w \in W$ and $c \in C_\Lambda^0$.  Then we have
\begin{equation*}
\pi^{qp}(x^{\mu}T_{w}) x^{c} = h^{c}(w) x^{\mu + wc}.
\end{equation*}
In particular, if $w \in W^{c}$ then we have
\begin{equation*}
\pi^{qp}(x^{\mu}T_{w}) x^{c} =  x^{\mu + wc}.
\end{equation*}
\end{theorem}
\begin{proof}
The first statement is \cite[Corollary 5.10(2)]{SSV2} (which holds for all $c \in \overline{C}^+$ with slightly more complicated root-system statistics if $\varphi(c) = 1$). It can also be shown directly by induction on $l(w)$, using the explicit formulas in Theorem \ref{thm:qp_rep} together with Proposition \ref{prop:ht}(1). The second statement follows from the first and \ref{prop:ht}(2).
\end{proof}

\begin{corollary}\label{cor:gammas}
Let $c \in C^0_{\Lambda}$.  For $h \in \widetilde{\mathcal{H}}_{\Lambda}$ and $w \in W^{c}$, we have
\begin{equation*}
\gamma_{w}^{qp}(\pi^{qp}(h) x^{c}) = \sum_{\substack{w' \in W \\ w'c = wc} } h^{c}(w') \gamma_{w'}(h)  = \sum_{u \in W_{c}}  t(u) \gamma_{wu}(h).
\end{equation*}
\end{corollary}
\begin{proof}
Write $h = \sum_{w \in W} \gamma_{w}(h) T_{w}$.  Then by Theorem \ref{thm:XT_formula}, 
\begin{equation*}
\pi^{qp}(h) x^{c} = \sum_{w \in W} \gamma_{w}(h) h^{c}(w) x^{w c}.
\end{equation*}
Taking $\gamma_{w}^{qp}$ of both sides gives the first equality; reparametrizing and using Proposition \ref{prop:ht} gives the second.
\end{proof}

\subsection{Quasi-polynomial generalizations of Macdonald polynomials in the $q \rightarrow \infty$ limit}\label{sec:connections_to_pols}
Let $c \in \overline{C}^{+} \subset E$.  In \cite{SSV2}, we defined a family of quasi-polynomials $E_{y} = E_{y}(q; \mathbf{t}; \tau) \in \mathbb{F}[\widetilde{\mathcal{O}}_{c}]$, indexed by $y \in \widetilde{\mathcal{O}}_{c}$, depending on parameters $q,\mathbf{t} = \{t_i \}$ and an additional torus parameter $\tau \in \text{Hom}(Q, \mathbb{F}^{\times})$ satisfying certain conditions.  We also defined
\begin{equation*}
E_{y}^{\pm} = \pi^{qp}(\mathbf{1}^{\pm}) E_{y} \in [\widetilde{\mathcal{O}}_{c}]
\end{equation*}
for $y \in \widetilde{\mathcal{O}}_{c}$ (see Definitions 6.32, 6.39 of \cite{SSV2}).  For $y \in \Lambda$ (i.e., $c = 0$), $E_{y}$, $E_{y}^{\pm}$ (resp.) are nonsymmetric and (anti-)symmetric Macdonald polynomials (see Remark 6.33 of \cite{SSV2}).

In \cite{SSV2}, we obtained several results for these quasi-polynomials in the $q \rightarrow \infty$ limit (under some regularity conditions on $\tau$ that we implicitly assume here); we use the notation $\overline{E}_{y} = \overline{E}_{y}(q;\mathbf{t}; \tau), \overline{E}_{y}^{\pm} =  \overline{E}_{y}^{\pm}(q; \mathbf{t}; \tau)$ for these objects. 

\begin{theorem}[Theorem 6.49 of \cite{SSV2}] \label{thm:qp_limit}
	Let $y \in \widetilde{\mathcal{O}}_{c}$ with $y = g_{y}^{-1} y_{-}$, where $g_{y}^{-1} \in W$ is minimal length and $y_{-}$ is anti-dominant.  Then we have
	\begin{equation*}
		\overline{E}_{y} = d(y; \mathbf{t}) \cdot \pi^{qp}(T_{g_{y}}^{-1}) x^{y_{-}}.
	\end{equation*}
	In particular, for $y \in \widetilde{\mathcal{O}}_{c}$ anti-dominant, $\overline{E}_{y} = x^{y}$ and $\overline{E}_{y}^{\pm} = \pi^{qp}(\mathbf{1}^{\pm}) x^{y}$.
\end{theorem}
The coefficient $d(y;\mathbf{t})$ is computed in \cite{SSV2}.

Let $\tilde{t}(w), \pi^{qp}(\widetilde{\mathbf{1}}^{\pm})$ (resp.) denote $t(w), \pi^{qp}(\mathbf{1}^{\pm})$ (resp.) defined with respect to $t_i^{-1}$-parameters.

\begin{corollary}\label{cor:qp1y_domin}
Let $y \in  \widetilde{\mathcal{O}}_{c}$ be dominant.  Then
\begin{equation*}
\pi^{qp}(\mathbf{1}^{\pm}) x^{y} =  \tilde{t}(w_0)^{\mp 2} \iota \overline{E}_{-y}^{\pm}(q; \mathbf{t}^{-1}; \tau)
\end{equation*}
\end{corollary}
\begin{proof}
Using Proposition \ref{prop:inv_params_qp} and Theorem \ref{thm:qp_limit}, we compute
\begin{equation*}
\iota \pi^{qp}(\mathbf{1}^{\pm}) \iota x^{-y} = \tilde{t}(w_0)^{\mp 2} \pi^{qp}(\tilde{\mathbf{1}}^{\pm}) x^{-y} = \tilde{t}(w_0)^{\mp 2}\overline{E}_{-y}^{\pm}(q; \mathbf{t}^{-1}; \tau) .
\end{equation*}
Taking $\iota$ of both sides gives the result.
\end{proof}

\subsection{Partially (anti-)symmetric Macdonald polynomials in the $q \rightarrow \infty$ limit}\label{sec:lim_ptl}
We now discuss some families of polynomials that appear in the results of this paper.
\begin{proposition}\label{prop:J_MD}
Let $\mu \in \Lambda$ dominant and $\hat{w} \in W$.  We have the following identities in $\mathbb{F}[\Lambda]$:
\begin{align*}
\pi(T_{\hat{w}^{-1}}) x^{\mu} &=   d(-\hat{w}^{-1} \mu;\mathbf{t}^{-1})^{-1} \cdot \iota \overline{E}_{\hat{w}^{-1} (-\mu)}(q;\mathbf{t}^{-1};\tau)\\
\pi(\mathbf{1}^{\pm}_{J} T_{\hat{w}^{-1}}) x^{\mu} &=  \tilde{t}(w_0(W_{J}))^{\mp 2}  d(-\hat{w}^{-1} \mu; \mathbf{t}^{-1})^{-1} \cdot \iota \pi(\widetilde{\mathbf{1}}_{J}^{\pm})  \overline{E}_{\hat{w}^{-1} (-\mu)}(q; \mathbf{t}^{-1};\tau) 
\end{align*}
where $\tilde{t}(\mu), \pi(\widetilde{\mathbf{1}}_{J}^{\pm})$ (resp.) denote $t(\mu), \pi(\mathbf{1}_{J}^{\pm})$ (resp.) defined with respect to $t_i^{-1}$-parameters.\end{proposition}
\begin{proof}
The first equality follows from Proposition \ref{prop:inv_params} and Theorem \ref{thm:qp_limit}.

For the $\mathbf{1}_{J}^{+}$ equality, we compute using Proposition \ref{prop:ptl1minus} and Definition \ref{def:partial_symm},
\begin{align*}
\iota \pi(\mathbf{1}^{+}_{J}) \iota &=t(w_0(W_{J}))^{2}\sum_{u \in W_J} t(u)^{-1} \iota \pi(T_{u^{-1}}^{-1} )\iota \\
&= t(w_0(W_{J}))^{2}\sum_{u \in W_J} \tilde{t}(u)  \pi(\widetilde{T}_{u} )\ \\
&= \tilde{t}(w_0(W_{J}))^{-2} \pi(\widetilde{\mathbf{1}}_{J}^{+}).
\end{align*}
The $\mathbf{1}_{J}^{-}$ case is analogous.
\end{proof}

We use the partial symmetrizer to define $J$-partially (anti-)symmetric Macdonald polynomials.  The following lemma is essential in indexing these objects.

\begin{lemma}\label{lem:Jptlsymmpols}
Let $J \subseteq [1, r]$ and $\lambda \in \Lambda$.  The following are equivalent:
\begin{enumerate}
\item $\alpha_{j}(\lambda) \geq 0$ for all $j \in J$.
\item There are (necessarily unique) $\mu \in \Lambda^{+}$, $\hat{w}^{-1} \in W^{J_{\mu}}$ such that $\lambda = \hat{w}^{-1} \mu$ and $\hat{w} \in W^{J}$.
\item There are $\mu \in \Lambda^{+}$ and $\hat{w} \in W^{J}$ such that $\lambda = \hat{w}^{-1} \mu$ and $\mu + \hat{w} c \in E^{+}$ where $c \in \overline{C^{+}}$, $\varphi(c) < 1$, and $J_{c} = J$ (this condition does not depend on choice of $c$).
\end{enumerate}
\end{lemma}

\begin{proof}
(1) $\Rightarrow$ (2): There are unique $\mu \in \Lambda^{+}$, $\hat{w}^{-1} \in W^{J_{\mu}}$ such that $\lambda = \hat{w}^{-1} \mu$.  Suppose for contradiction that $\hat{w} \not \in W^{J}$.  Then $\exists \alpha \in \Phi_{J}^{+}$ such that $\beta = \hat{w} \alpha \in \Phi^{-}$.  By (1), we have $\alpha(\lambda) = \alpha(\hat{w}^{-1} \mu) = \beta(\mu) \geq 0$.  Since $\mu \in \Lambda^{+}$ and $\beta \in \Phi^{-}$, we must have $\beta(\mu) = 0$, which implies $\beta \in \Phi_{J_{\mu}}^{-}$.  But then $\hat{w}^{-1} \beta = \alpha \in \Phi^{+}$ contradicts $\hat{w}^{-1} \in W^{J_{\mu}}$.

(2) $\Rightarrow$ (3): Suppose $\mu \in \Lambda^{+}$, $\hat{w}^{-1} \in W^{J_{\mu}}$, $\hat{w} \in W^{J}$ and $c \in C^{+}$ with $\varphi(c) < 1$ and $J_{c} = J$.  We need to show $\mu + \hat{w} c \in E^{+}$.  Suppose not, then $\exists \alpha \in \Phi^{+}$ such that $\alpha(\mu + \hat{w} c) = \alpha(\mu) + (\hat{w}^{-1} \alpha)(c) < 0$.  Since $\varphi(c) < 1$, this implies $|(\hat{w}^{-1} \alpha)(c)| < 1$.  So $\alpha(\mu) = 0$ since we can't have $\alpha(\mu) \geq 1$.  So $\hat{w}^{-1} \alpha \in \Phi^{-} \setminus \Phi^{-}_{J}$ and $\alpha \in \Phi_{J_{\mu}}^{+}$.  But this contradicts $\hat{w}^{-1} \in W^{J_{\mu}}$.

(3) $\Rightarrow$ (1): Suppose $\lambda = \hat{w}^{-1} \mu$, where $\hat{w} \in W^{J}$, $\mu \in \Lambda^{+}$ and $\mu + \hat{w} c \in E^{+}$.  Let $\alpha \in \Phi_{J}^{+}$ then $\alpha(\lambda) = \alpha(\hat{w}^{-1} \mu) = (\hat{w} \alpha)(\mu) \geq 0$ since $\hat{w} \alpha \in \Phi^{+}$ (as $\hat{w} \in W^{J}$) and $\mu \in \Lambda^{+}$.
\end{proof}

\begin{definition}
If $\lambda \in \Lambda$ satisfies the conditions of Lemma \ref{lem:Jptlsymmpols}, we say it is $J$-dominant.
\end{definition}

\begin{remark}
The definition above coincides with \cite{BDF} for $GL_{r}$, which uses condition (1) of Lemma \ref{lem:Jptlsymmpols}.
\end{remark}

\begin{definition}\label{def:MDptl}
Let $J \subseteq [1,r]$ and let $\lambda \in \Lambda$ be $J$-dominant with $\lambda = \hat{w}^{-1} \mu$ where $\mu \in \Lambda^{+}$ and $\hat{w}^{-1} \in W^{J_{\mu}}$.  We define the polynomials
\begin{align*}
p_{\lambda}^{J, +}(\mathbf{t}) &:= \pi(\mathbf{1}_{J}^{+} T_{\hat{w}^{-1}}) x^{\mu} \\
&=  \tilde{t}(w_0(W_{J}))^{-2} d(-\hat{w}^{-1}\mu; \mathbf{t}^{-1})^{-1}  \cdot \iota \pi(\widetilde{\mathbf{1}}_{J}^{+})  \overline{E}_{\hat{w}^{-1} (-\mu)}(q; \mathbf{t}^{-1})  \in \mathbb{F}[\Lambda] \\
p_{\lambda}^{J, -}(\mathbf{t}) &:= \pi(\mathbf{1}_{J}^{-} T_{\hat{w}}^{-1}) \iota x^{\mu}=   \pi(\mathbf{1}_{J}^{-})  \overline{E}_{\hat{w} (-\mu)}(q; \mathbf{t})  \in \mathbb{F}[\Lambda]. 
\end{align*}
\end{definition}

\begin{remark}
In the $GL_{r}$-setting, up to a normalization factor, the polynomials $p^{ J, \pm}_{\lambda}$ were introduced by Baker, Dunkl, and Forrester \cite{BDF} at the $q$-level, using $E$ instead of $\overline{E}$ in the definition above (up to applying $\iota$ and inverting $t$-parameters).  In fact, they consider polynomials that are symmetric and antisymmetric with respect to two prescribed sets of variables, but this will not be needed here.  These objects were investigated further by Marshall \cite{Mar} and Baratta \cite{Bar}.  More recently,  in the $GL_{r}$-setting the polynomials $p^{ J, -}_{\mu}$,  for $\mu \in \Lambda^{+}$ (i.e., $\hat{w} = 1$), were related to $p$-adic parahoric Whittaker functions, see Proposition 9.5 of \cite{BBBG19}.  In another recent work \cite{ABW}, the partially anti-symmetric polynomials $p_{\lambda}^{J,-}$ were obtained as partition functions for certain vertex models studied in that paper.
\end{remark}

%%%%%%%%%%%%%%%%%%%%%%%%%%%%%%%%
\subsection{Metaplectic representations} \label{sec:metreps}

In this section, we recall results from \cite{SSV} on metaplectic representations that will be used throughout the rest of the paper.  We also define coefficient functions of metaplectic polynomials, and connect them to coefficient functions of Hecke algebra elements (see Corollary \ref{met_coeffs}).

Fix a metaplectic parameter $n \in \mathbb{Z}_{\geq 1}$ and lattice $\Lambda \in \mathcal{L}$.  Let $\mathbf{Q}: \Lambda \rightarrow\mathbb{Q}$ be a non-zero $W$-invariant quadratic form that restricts to an integer-valued quadratic form on $Q^\vee$. Write $\mathbf{B}: \Lambda \times \Lambda \rightarrow \mathbb{Q}$ for the $W$-invariant symmetric bilinear pairing
\begin{equation*}
\mathbf{B}(\lambda,\mu):=\mathbf{Q}(\lambda+\mu)-\mathbf{Q}(\lambda)-\mathbf{Q}(\mu)\qquad (\lambda,\mu\in \Lambda).
\end{equation*}
Then 
\begin{equation}\label{eqn:BQ}
	\mathbf{B}(\lambda,\alpha^\vee)=\mathbf{Q}(\alpha^\vee)\alpha(\lambda)\qquad (\lambda\in \Lambda,\, \alpha\in \Phi)
\end{equation}
and
\begin{equation}\label{def:m}
	m(\alpha):=\frac{n}{\textup{gcd}(n,\mathbf{Q}(\alpha^\vee))}=\frac{\textup{lcm}
		(n,\mathbf{Q}(\alpha^\vee))}{\mathbf{Q}(\alpha^\vee)} \qquad (\alpha\in\Phi)
\end{equation}
defines a $W$-invariant $\mathbb{Z}_{>0}$-valued function on $\Phi$. The associated metaplectic root system is defined by
\begin{equation*}
\Phi^m:=\left\{\alpha^m\,\, | \,\,
\alpha\in\Phi\right\}\,\,\textup{ with }\,\, \alpha^m:=m(\alpha)^{-1}\alpha.
\end{equation*}
Then $\Phi^m = m\cdot \Phi \simeq \Phi$ if $m$ is constant, and $\Phi^m\simeq\Phi^\vee$ otherwise (see \cite{SSV}). In particular, the Weyl group of $\Phi^m$ is still $W$. 
The fixed basis $\{\alpha_1,\ldots,\alpha_r\}$ of $\Phi$ determines a basis $\{\alpha_1^m,\ldots,\alpha_r^m\}$ of $\Phi^m$.
Note that
$\Phi^{m\vee}=\{\alpha^{m\vee}\}_{\alpha\in\Phi}$, with $\alpha^{m\vee}=m(\alpha)\alpha^\vee$. Write 
\begin{equation*}
Q^{m\vee}:=\mathbb{Z}\Phi^{m\vee}\subseteq Q^\vee
\end{equation*}
for the coroot lattice of $\Phi^m$.
Let $\Lambda \in \mathcal{L}$ (note that $\mathcal{L}$ is defined with respect to $\Phi$, not $\Phi^m$) and
\begin{equation*}
\Lambda^{m} = \{\lambda \in \Lambda : m(\alpha) | \alpha(\lambda) \text{ for all } \alpha \in \Phi \}.
\end{equation*}
Then $\widetilde{\mathcal{H}}_{\Lambda^m}$ is the $\Lambda^m$-extended affine Hecke algebra associated to $\Phi^m$.  It has generators $T_1, \dots, T_r$ and $x^{\lambda}$ for $\lambda \in \Lambda^m$.

The alcove
\begin{align}\label{Cplus_m}
C^{m+}& :=\{ y\in E \,\, | \,\, 0<\alpha(y)< m(\alpha) \quad \forall\, \alpha\in\Phi^+ \} \\ &= \{y \in E  \,\, | \,\, 0 < \alpha_i (y) \text{ for } 1 \leq i \leq r \text{ and } \varphi(y) < m(\varphi) \}
\end{align}
is the metaplectic fundamental alcove in $E$.

In analogy with Definition \ref{def:key_ccondit}, we have the following.

\begin{definition}\label{def:key_ccondit_met}
For $\Lambda \in \mathcal{L}$, we define
\begin{equation*}
C^0_{\Lambda^m} = \{ c \in \overline{C^{m+}} \cap \Lambda : \widetilde{W}_{\Lambda^m, c} = W_{c} \}.
\end{equation*}
\end{definition}

For $m, x \in \mathbb{Z}_{\geq 1}$ let $r_{m}(x) = (x \mod m) \in \{0, 1, \dots, m-1 \}$ be the remainder function.

We define the \textit{metaplectic divided difference} $\nabla_i^m: \mathbb{F}[\Lambda] \to \mathbb{F}[\Lambda]$, $1 \leq i \leq r$, by
\begin{equation*}
	\nabla_i^m(x^\lambda) :=\Bigl(\frac{1-x^{(r_{m(\alpha_i)}(\alpha_i(\lambda))-\alpha_i(\lambda))\alpha_i^\vee}}{1-x^{m(\alpha_i)\alpha_i^\vee}}\Bigr)x^\lambda
\end{equation*}
for $\lambda \in \Lambda$.  Note that, by a geometric series expansion,
\begin{equation*}
\Bigl(\frac{1-x^{(r_{m(\alpha_i)}(\alpha_i(\lambda))-\alpha_i(\lambda))\alpha_i^\vee}}{1-x^{m(\alpha_i)\alpha_i^\vee}}\Bigr) \in \mathbb{F}[\Lambda^m].
\end{equation*}

\begin{definition}[Representation parameters]\label{RepPar2}
	Let $g_j(\alpha)\in\mathbb{F}^\times$ for $j \in\mathbb{Z}$ and $\alpha \in \Phi$ be parameters satisfying
	the following conditions:
	\begin{enumerate}
		\item $g_j(\alpha)=-1$ if $j\in n\mathbb{Z}$,
		\item $g_{j+n}(w\alpha)=g_{n}(\alpha) \qquad (\forall w \in W)$,
		\item $g_j(\alpha)g_{n-j}(\alpha)=t_\alpha^{-2}$ if $j\in\mathbb{Z}\setminus n\mathbb{Z}$.
	\end{enumerate}
\end{definition}

%%%%%%%%%%%%%%%%%%%%%%%%%%%%%%%%%%%%%%%%%%%%

We recall the metaplectic representation of $\widetilde{\mathcal{H}}_{\Lambda^m}$ from \cite{SSV} (see also \cite{CGP}).

\begin{theorem}\label{metaplectic_repr}
	The formulas
	\begin{equation}\label{formulasTHM}
		\begin{split}
			\pi^m(T_i)x^\lambda   &:= -t_i g_{-\mathbf{B}(\lambda,\alpha_i^\vee)}(\alpha_i) - (t_i-t_i^{-1})\nabla_i^m(x^\lambda) \\
			\pi^m(X^\mu)x^\lambda &:= x^{\lambda+\nu}
		\end{split}
	\end{equation}
	for $\lambda\in \Lambda$, $i=1,\ldots,r$ and $\mu\in \Lambda$ turn $\mathbb{F}[\Lambda]$ into a left $\widetilde{\mathcal{H}}_{\Lambda^m}$-module. 
\end{theorem}
%%%%%%%%%%%%%%%%%%%%%%%%%%%%%%%%%%%%%%%%%%%%

\begin{remark}\label{met_nonmet} \hfill
	\begin{enumerate}
		\item Note that $\mathbb{F}[\widetilde{\mathcal{O}}_{\Lambda^{m}, c}]$, for $c \in \overline{C^{m+}} \cap \Lambda$, is an invariant subspace of $\pi^{m}$.
		\item In the non-metaplectic case, $n=1$, the representation $\pi^{m}$ reduces to the polynomial representation $\pi$ of Theorem \ref{thm:pol_rep}.  
		\item The metaplectic representation $\pi^m$ can be identified with the quasi-polynomial representation $\pi^{qp}$ of $\widetilde{H}_{\Lambda^m}$ restricted to $\mathbb{F}[\Lambda] \subset \mathbb{F}[E]$, see \cite[Section 9]{SSV2}.
	\end{enumerate}
\end{remark}

We have the following crucial relation between $\pi^{m}$ restricted to the subspace $\mathbb{F}[\Lambda^m]$ and the polynomial representation $\pi$ of Theorem \ref{thm:pol_rep}, in the case where $m(\alpha)$ is constant.

\begin{proposition}\label{prop:pim_restr_pi}
Suppose $m(\alpha)$ is constant, $m(\alpha) = m$ for all $\alpha \in \Phi$. Let $d_{m} : \mathbb{F}[\Lambda] \rightarrow \mathbb{F}[\Lambda^m]$ be the linear extension $x^{\lambda} \rightarrow x^{m\lambda}$ for $\lambda \in \Lambda$. Then for $w \in W$ we have
\begin{equation*}
	\pi^{m}(T_w) \circ d_m = d_m \circ \pi(T_w).
\end{equation*} 
\end{proposition}

\begin{proof}	
Follows from the formulas in Theorem \ref{metaplectic_repr} and Theorem \ref{thm:pol_rep}.
\end{proof}

We now define $\gamma$-coefficient functions analogous to Definition \ref{def:gamma_coeff_qp} in the quasi-polynomial context.
\begin{definition}\label{def:gamma_coeff_met}
Let $c \in C^0_{\Lambda^m}$ and $w \in W^{c}$.  Define coefficient functions $\gamma_{w,c}^{m} = \gamma_{w}^{m}: \mathbb{F}[\widetilde{\mathcal{O}}_{\Lambda^{m}, c }] \rightarrow \mathbb{F}[\Lambda^{m}]$ by
\begin{equation*}
\gamma_{w}^{m}\big( \sum_{w' \in W^{c }} p_{w'} \cdot x^{w' c }\big) = p_{w},
\end{equation*}
where $p_{w'} \in \mathbb{F}[\Lambda^{m}]$ for $w' \in W^{c}$.  
\end{definition}

\begin{lemma}\label{lem:iota_gamma_met}
Let $c \in C^0_{\Lambda^m}$, $w \in W^{c}$, and $f \in \mathbb{F}[\widetilde{\mathcal{O}}_{\Lambda^{m}, c }]$.  We have 
\begin{equation*}
\iota \gamma_{w,c}^{m}(f) = \gamma^{m}_{(ww_0)^{-w_0 c}, -w_0 c}(\iota f).
\end{equation*}
\end{lemma}
\begin{proof}
Write $f = \sum_{w' \in W^{c }} p_{w'} \cdot x^{w' c }$, so $\iota \gamma_{w,c}^{m}(f) = \iota p_{w}$.  On the other hand
\begin{equation*}
\iota f = \sum_{w' \in W^{c}} (\iota p_{w'}) \cdot x^{-w'c} = \sum_{w' \in W^{c}} (\iota p_{w'}) \cdot x^{w' w_0 (-w_0 c)}.
\end{equation*}
Note that $-w_0 c \in C^{0}_{\Lambda^{m}}$, so $\iota \gamma^{m}_{w,c}(f) = \gamma^{m}_{(ww_0)^{-w_0c}, -w_0 c}(\iota f)$.
\end{proof}

\begin{definition}\label{def:met_h_stats}
For $w \in W$ and $c \in C_{\Lambda^m}^0$, we define root system statistics $h^{m,c}(w)$ by
\begin{equation*}
	h^{m, c}(w) = (-1)^{l(w)}t(w)\prod_{\alpha \in \Pi(w)} g_{-B(c, \alpha^\vee)}(\alpha).
\end{equation*}
\end{definition}

These statistics are related to those defined in the quasi-polynomial context \eqref{eq:qp_h_stats} via the identifications in \cite[Section 9]{SSV2}.
\begin{proposition}\label{prop:ht_met}
Let $c \in C_{\Lambda^m}^0$, then the root system statistics $h^{m,c}(w)$ satisfy the following properties:
\begin{enumerate}
	\item If $w \in W$ and $l(s_iw) > w$ then 
	\begin{equation*}
		h^{m,c}(s_i w) = -t_i g_{-B(wc, \alpha_i^\vee)}(\alpha_i) h^c(w)
	\end{equation*}
	\item If $v \in W^c$ and $u \in W_c$, then we have
	\begin{equation*}
		h^{m,c}(vu) = h^{m,c}(v) t(u).
	\end{equation*}
\end{enumerate}
\end{proposition}
\begin{proof}
These follow from standard root-system manipulations.
\end{proof}

\begin{theorem}\label{thm:XT_met}
Let $c \in C_{\Lambda^m}^0$, $\mu \in \Lambda^{m}$, and $w \in W$.  Then
\begin{equation*}
\pi^{m}(x^{\mu} T_{w}) x^{c} = h^{m, c}(w) x^{\mu + w c}.
\end{equation*}
\end{theorem}
\begin{proof}
This follows from Theorem \ref{thm:XT_formula} and the identifications in \cite[Section 9]{SSV2}. It can also be proved directly using induction on $l(w)$ as follows. The result is trivial for $w = 1$, now suppose that $w \in W$ and $l(s_iw) > l(w)$ for some $1 \leq i \leq r$. Then we have $0 \leq \alpha_i(wc) < m(\alpha)$ (the latter uses $c \in C_{\Lambda^m}^0$). It then follows from the formulas in Theorem \ref{metaplectic_repr} and the induction hypothesis that
\begin{equation*}
	\pi^m(T_{s_iw})x^c = h^{m, c}(w)\pi^m(T_i)x^{wc} = -t_ig_{-B(wc, \alpha_i^\vee)} h^{m, c}(w) x^{s_iw c},
\end{equation*}
and the result then follows from Proposition \ref{prop:ht_met}(1).
\end{proof}

\begin{corollary}\label{met_coeffs}
Let $c \in C^0_{\Lambda^m}$.  Let $w \in W^{c}$ and $h \in \widetilde{\mathcal{H}}_{\Lambda^{m}}$, then we have
\begin{equation}
	\gamma^m_{w, c}(\pi^m(h) x^c) = h^{m,c}(w)\sum_{u \in W_c} t(u) \gamma_{wu}(h).
\end{equation}
\end{corollary}

\subsection{Whittaker functions}\label{sec:Whittaker}
We briefly recall the connection between the representations discussed so far and (metaplectic) Whittaker functions.  

We recall the connection between the representation $\pi^{m}$, the metaplectic Demazure-Lusztig operators, and metaplectic Whittaker functions of \cite{CGP}, we will follow \cite{SSV, SSV2}.

\begin{definition}\label{def:met_specialization}
We will write $\pi^{m}_{v}$ for the metaplectic representation $\pi^{m}$ with all Hecke parameters $t_i := v^{1/2}$ and equal representation parameters, $g_{j} := g_{j}(\alpha) $ for all $\alpha \in \Phi$, $j \in \mathbb{Z}$.  In \cite{SSV, SSV2}, this is referred to as the ``equal Hecke and parameter case''.  We will also write $h^{m,c}_{v}$ to denote the statistics in Definition \ref{def:met_h_stats} under this particular specialization.  We will suppress the dependence on $v$ when it is clear from context.
\end{definition}

\begin{remark}\label{met_nonmet_v}
In the non-metaplectic case, $n=1$, the representation $\pi^{m}_{v}$ reduces to the polynomial representation $\pi_{v}$ of Definition \ref{def:poly_rep_v}.  
\end{remark}

\begin{definition}\cite{CGP, SSV}\label{prop:IandW}
Define the metaplectic operators $\mathcal{T}_{i, v}^{m} \in \mathrm{End}(\mathbb{F}[\Lambda])$ as in (4.5) of \cite{SSV} by
\begin{equation*}
\mathcal{T}_{i, v}^{m} = -v^{1/2} \mathbf{y}^{\rho} \pi^{m}_{v}(T_i^{-1}) \mathbf{y}^{-\rho},
\end{equation*}
for $1 \leq i \leq r$, where we have set $x = \mathbf{y}$ to match the notation in the rest of this paper for Whittaker functions.  For $w = s_{i_1} \cdots s_{i_r} \in W$ a reduced expression, we write
\begin{equation*}
\mathcal{T}_{w,v}^{m} = \mathcal{T}_{i_1,v}^{m} \cdots \mathcal{T}_{i_r,v}^{m} \in \mathrm{End}(\mathbb{F}[\Lambda]).
\end{equation*}
For $\lambda \in \Lambda^{+}$ and $w \in W$, define the polynomials
\begin{align*}
\widetilde{\mathcal{I}}_{w, \lambda}^{m}(\mathbf{y}; v)& := \mathcal{T}_{w,v}^{m}(\mathbf{y}^{w_0 \lambda}) = (-v^{1/2})^{l(w)} \mathbf{y}^{\rho} \pi^{m}_{v}(T_{w^{-1}}^{-1}) \mathbf{y}^{-\rho+ w_0 \lambda}\\
\widetilde{\mathcal{W}}_{\lambda}^{m}(\mathbf{y}; v) &:= \sum_{w \in W} \mathcal{T}_{w,v}^{m}(\mathbf{y}^{w_0 \lambda}) =  v^{l(w_0)} \mathbf{y}^{\rho} \pi^{m}_{v}(\mathbf{1}^{-}) \mathbf{y}^{-\rho + w_0 \lambda}. 
\end{align*}
We also define the closely-related function
\begin{equation*}
\widetilde{\mathbb{W}}_{\lambda}^{m}(\mathbf{y};v) := \sum_{w \in W} \iota \mathcal{T}_{w,v}^{m} \iota (\mathbf{y}^{\lambda}),
\end{equation*}
so that $\widetilde{\mathbb{W}}_{\lambda}^{m}(\mathbf{y}) = \iota \widetilde{\mathcal{W}}_{-w_0 \lambda}^{m}(\mathbf{y})$.  There is also the following relation:
\begin{align*}
w_0 \widetilde{\mathcal{W}}_{\lambda}^{m}(\mathbf{y}) & =v^{l(w_0)} \mathbf{y}^{-\rho} \iota w_0 \iota \pi^{m}(\mathbf{1}^{-}) \mathbf{y}^{-\rho + w_0 \lambda} \\
&= v^{l(w_0)} \mathbf{y}^{-\rho} \iota  \pi^{m}(\mathbf{1}^{-}) \mathbf{y}^{-\rho - \lambda} \\
&= \iota \widetilde{\mathcal{W}}_{-w_0 \lambda}^{m}(\mathbf{y}) \\
&= \widetilde{\mathbb{W}}_{\lambda}^{m}(\mathbf{y}).
\end{align*}

\end{definition}

\begin{remark}
The operators $\mathcal{T}_{w}^{m}$ are the metaplectic Demazure-Lusztig operators introduced in \cite{CGP}.  It is shown in \cite{CGP} (resp. \cite{PP}) that $\widetilde{\mathcal{W}}_{\lambda}^{m}(\mathbf{y}) $ (resp. $\widetilde{\mathcal{I}}_{w, \lambda}^{m}(\mathbf{y})$) are the metaplectic spherical (resp. Iwahori) Whittaker functions.  These objects have also been investigated in many other recent papers, see e.g., \cite{BBBG19, BBBG20, BBBG21, PP2}.
\end{remark}

Using Definition \ref{prop:IandW}, Remark \ref{met_nonmet}, and Proposition \ref{prop:TDW_antisymm}, there is the following formula for the \textit{non-metaplectic} spherical Whittaker functions in terms of the Demazure-Whittaker operators (recall Definition \ref{def:DW_opers}):
\begin{equation}\label{eq:sphW_nonmet}
\widetilde{\mathcal{W}}_{\lambda}(\mathbf{y}) := \widetilde{\mathcal{W}}_{\lambda}^{1}(\mathbf{y}; v) = v^{l(w_0)} \mathbf{y}^{\rho} \pi_{v}(\mathbf{1}^{-}) \mathbf{y}^{-\rho + w_0 \lambda} = v^{l(w_0)} \mathbf{y}^{\rho} \sum_{w \in W} \mathcal{T}_{w, v^{1/2}} \mathbf{y}^{-\rho + w_0 \lambda},
\end{equation}
for $\lambda \in \Lambda^{+}$.
Note that, in $GL_{r+1}$, \eqref{eq:sphW_nonmet} may be rewritten using $\rho_{GL} = (r-1, r-2, \dots, 0)$
\begin{equation}\label{eq:sphW_nonmet_typeA}
 \widetilde{\mathcal{W}}_{\lambda}(\mathbf{y}) = v^{l(w_0)} \mathbf{y}^{\rho_{GL}} \sum_{w \in W} \mathcal{T}_{w, v^{1/2}} \mathbf{y}^{-\rho_{GL} + w_0 \lambda},
\end{equation}
since $\rho - \rho_{GL}$ is constant.  

\begin{remark}\label{rem:met_whitt_special}
Suppose $\lambda \in \Lambda^{m}$.  Then using the fact that $\pi^{m}_{v}$ restricts to the polynomial representation on the subspace $\mathbb{F}[\Lambda^m]$ \cite{SSV2}, we have
\begin{align*}
\mathbf{y}^{\rho^{m} - \rho} \widetilde{\mathcal{W}}_{\lambda - \rho}^{m}(\mathbf{y}) = \widetilde{\mathcal{W}}_{\lambda - \rho^m, \Phi^{m}}(\mathbf{y}).
\end{align*}
In the case when $m(\alpha) = m$ for all $\alpha$, we have $ \widetilde{\mathcal{W}}_{\lambda - \rho^m, \Phi^{m}}(\mathbf{y}) = \widetilde{\mathcal{W}}_{\lambda/m - \rho, \Phi}(\mathbf{y}^{m})$.
\end{remark}

Recall the (component) metaplectic Whittaker functions $\tilde{\phi}_{\theta}^{o}$ introduced in \cite{BBBG20}.  These depend on a choice of coset representatives for $\Lambda / \Lambda^m$, which we take to be $\{\rho - wc: c \in \widetilde{C}_{\Lambda^m}^0, w \in W^c\}$, where we have fixed a choice of representatives $\widetilde{C}_{\Lambda^m}^0$ for $C_{\Lambda^m}^0 / \Omega_{\Lambda^m}$. Note $\widetilde{C}_{Q^{m\vee}}^0 = {C}_{Q^{m\vee}}^0$. In the $GL_r$ setting we take the representatives to be 
\begin{equation}\label{eq:glr_cosets}
\widetilde{C}_{\Lambda^m}^0 := \{(c_1, \dotsc, c_r): m > c_1 \geq c_2 \geq \dotsb \geq c_r \geq 0\},
\end{equation}
in which case our representatives agree with \cite[(4.1)]{BBBG20}. For other types these may be an incomplete set of representatives, we can extend arbitrarily as all of our computations and results will be restricted to the $\widetilde{W}_{\Lambda^m}$-span of $\widetilde{C}_{\Lambda^m}^0$. For $\theta \in W \widetilde{C}_{\Lambda^m}^0$, we set $\tilde\phi_\theta^o = \tilde \phi_{\rho - \theta}^o$.

We relate these objects to the coefficients $\gamma_{w}^{m}$ of Definition \ref{def:gamma_coeff_met}. This will be used in Section \ref{sec:met_results} where we translate our results on antisymmetric quasi-polynomials to the metaplectic setting.  

\begin{proposition}\label{thm:phi_to_gamma}
Let $\mu - \rho \in \Lambda^{+}$ with decomposition $w_0 \mu = \eta + \hat{w} c$ for some $\eta \in \Lambda^{m}$, $\hat{w} \in W^{c}$, $c \in \widetilde{C}^{0}_{\Lambda^m}$ and let $\theta \in W \cdot \widetilde{C}^{0}_{\Lambda^m}$.   Also let $\rho' \in \Lambda$ such that $\rho - \rho' \in E_{\textup{co}}$.  Then
\begin{equation*}
\mathbf{y}^{\rho' - \theta} \tilde{\phi}_{\theta}^{o}(\mathbf{y}; \varpi^{\rho' - \mu}; v) = v^{l(w_0)}  \iota \gamma_{w_0 \tilde w^{c} w_0 , -w_0c}^{m}\big( \pi^{m}_{v}(\mathbf{1}^{-}) \mathbf{y}^{- \mu}\big),
\end{equation*}
 if $w_0 \theta = \tilde{w} c$ for some $\tilde{w} \in W$ with $w_0 \tilde{w} \in W^{c}$.  Otherwise it is equal to zero.  In particular, it is always a polynomial in $\mathbb{F}[\Lambda^m]$.  
\end{proposition}

\begin{proof}
We have
\begin{equation}\label{eq:phi_theta}
 \tilde{\phi}_{\theta}^{o}(\mathbf{y}; \varpi^{\rho' - \mu}; v) = \Big[ \sum_{w \in W} \iota \mathcal{T}_{w,v}^{m} \iota \mathbf{y}^{\mu - \rho'} \Big]^{\rho' - \theta}
\end{equation}
where $f(\mathbf{y})^{\gamma}  \in \mathbf{y}^{-\gamma} \mathbb{F}[\Lambda^m]$ are the component functions defined in \cite{BBBG20}.  Note that the operators $\mathbb{T}_i$ are used there (see (3.49) of \cite{BBBG20}), but it is an easy check that $\mathbb{T}_i = \iota \mathcal{T}_{i}^{m} \iota$.

Since 
\begin{equation*}
\iota \mathcal{T}_{w,v}^{m} \iota = (-1)^{l(w)} v^{l(w)/2} \iota x^{\rho} \pi^{m}_{v}(T_{w^{-1}}^{-1}) x^{-\rho} \iota, 
\end{equation*}
using \eqref{1plus1minus_alt}, we have
\begin{equation*}
\sum_{w \in W} \iota \mathcal{T}_{w,v}^{m} \iota \mathbf{y}^{\mu - \rho'} = v^{l(w_0)}\iota \mathbf{y}^{\rho} \pi^{m}_{v}(\mathbf{1}^{-} )\mathbf{y}^{-\mu + \rho' - \rho}. 
\end{equation*}
Since $\rho - \rho' \in E_{\textup{co}}$, we may replace $\rho$ in the expression above by $\rho'$.  So the expression above is equal to
\begin{align}\label{eq:phi_theta_to_pi}
v^{l(w_0)}   \mathbf{y}^{-\rho' } \iota \pi^{m}_{v}(\mathbf{1}^{-} )\mathbf{y}^{-\mu }.
\end{align}

We have $- \mu = -w_0 \eta  - \tilde w w_0 c$, where $w_0 \hat{w} = \tilde w w_0$.  Now $-w_0 c \in \widetilde{C}^{0}_{\Lambda^m}$, so $ \pi^{m}_{v}(\mathbf{1}^{-} )\mathbf{y}^{-\mu } \in  \mathbb{F}[\widetilde{\mathcal{O}}_{\Lambda^m,  -w_0c}]$.  Thus $\iota \pi^{m}_{v}(\mathbf{1}^{-} )\mathbf{y}^{-\mu } \in  \mathbb{F}[\widetilde{\mathcal{O}}_{\Lambda^m,  c}]$.  So for \eqref{eq:phi_theta} to be non-vanishing we must have $\theta = w' c $ for some $w' \in W^{c}$. If this condition is satisfied, then it follows from the definitions that \eqref{eq:phi_theta} is equal to
\begin{equation*}
v^{l(w_0)} \mathbf{y}^{-\rho'  + \theta} \gamma_{w', c}^{m}\big( \iota \pi^{m}_{v}(\mathbf{1}^{-}) \mathbf{y}^{- \mu}\big).
\end{equation*}
The result then follows from Lemma \ref{lem:iota_gamma_met} and Lemma \ref{lem:w0_conj}.
\end{proof}

As discussed in \cite{BBBG21}, $\sum_{\theta} \tilde{\phi}_{\theta}^{o}$ is the metaplectic spherical Whittaker function $\widetilde{\mathcal{W}}^{m}$.  From the previous proposition, we obtain the following relationship between these objects and the coefficient functions $\gamma^{m}$ of antisymmetric polynomials $\pi^{m}_{v}(\mathbf{1}^{-}) \mathbf{y}^{- \mu}$.

\begin{corollary}\label{phitheta_metsph}
Let $\mu - \rho \in \Lambda^{+}$ with decomposition $w_0 \mu = \eta + \hat{w} c$ for some $\eta \in \Lambda^{m}$, $\hat{w} \in W^{c}$, $c \in \widetilde{C}^{0}_{\Lambda^m}$.   Also let $\rho' \in \Lambda$ such that $\rho - \rho' \in E_{\textup{co}}$.  Then
\begin{align*}
\sum_{w_0 \tilde w \in W^{c}} \mathbf{y}^{\rho'}  \tilde{\phi}_{w_0 \tilde w c}^{o}(\mathbf{y}; \varpi^{\rho' - \mu}; v) &= 
 v^{l(w_0)} \sum_{w_0 \tilde w \in W^{c}} \mathbf{y}^{w_0 \tilde w c} \iota \gamma_{w_0 \tilde w^{c} w_0 , -w_0c}^{m}\big( \pi^{m}_{v}(\mathbf{1}^{-}) \mathbf{y}^{- \mu}\big) \\ &= \mathbf{y}^{ \rho'} w_0 \widetilde{\mathcal{W}}_{\mu - \rho'}^{m}(\mathbf{y};v) \\ &= \mathbf{y}^{ \rho'} \widetilde{\mathbb{W}}_{\mu - \rho'}^{m}(\mathbf{y};v).
\end{align*}

\end{corollary}
\begin{proof}
The first equality follows from Proposition \ref{thm:phi_to_gamma}, by summing both sides over $W^{c}$ and multiplying by the appropriate monomial.  The second and third equalities follows from Lemma \ref{lem:iota_gamma_met}, as well as
\begin{align*}
 \sum_{w_0 \tilde w \in W^{c}} \mathbf{y}^{w_0 \tilde w c} \gamma^{m}_{w_0 \tilde w, c}\big( \iota \pi^{m}_{v}(\mathbf{1}^{-}) \mathbf{y}^{- \mu}\big) &=w_0 \iota w_0 \pi^{m}_{v}(\mathbf{1}^{-}) \mathbf{y}^{- \mu} \\
 &= w_0 \pi^{m}_{v}(\mathbf{1}^{-}) \mathbf{y}^{w_0 \mu} 
\end{align*}
and lastly Definition \ref{prop:IandW}.  See also Section 2.2 of \cite{BBBG21} for the very last statement (noting that their sum is over all $\theta$ instead of the restricted sum above).
\end{proof}

\begin{remark}
Consider the special case of Corollary \ref{phitheta_metsph} when $c \in E_{\textup{co}}$.  In this case, we obtain a \textit{non-metaplectic} spherical Whittaker function.  Indeed, we have $\mu \in \Lambda^m$, $W_{c} = W$ and $W^{c} = 1$.  Using Remark \ref{rem:met_whitt_special}, we have 
\begin{equation*}
 \mathbf{y}^{\rho}  \tilde{\phi}_{ c}^{o}(\mathbf{y}; \varpi^{\rho - \mu}) =w_0 \mathbf{y}^{ -\rho} \widetilde{\mathcal{W}}_{\mu - \rho}^{m}(\mathbf{y}) = w_0\mathbf{y}^{-\rho^m} \widetilde{\mathcal{W}}_{\mu - \rho^{m}, \Phi^{m}}(\mathbf{y}) = \mathbf{y}^{\rho^m} \widetilde{\mathbb{W}}_{\mu - \rho^{m}, \Phi^{m}}(\mathbf{y}).
\end{equation*}
We also have
\begin{align*}
w_0 \Big(  \mathbf{y}^{\rho}  \tilde{\phi}_{ c}^{o}(\mathbf{y}; \varpi^{\rho - \mu}) \Big) &= \mathbf{y}^{-\rho^m} \widetilde{\mathcal{W}}_{\mu - \rho^{m}, \Phi^{m}}(\mathbf{y}) =  \mathbf{y}^{-\rho^m} w_0 \widetilde{\mathbb{W}}_{\mu - \rho^{m}, \Phi^{m}}(\mathbf{y}).
\end{align*}
For $GL_r$ this is Theorem E of \cite{BBBG21}, see also the second case of Remark \ref{rem:met_special}.
\end{remark}

By Corollary \ref{phitheta_metsph}, to compute $\widetilde{\mathcal{W}}^{m}_{\mu - \rho'}$, it is sufficient to compute $\tilde{\phi}_{w_0 \tilde w c}^{o}$ for all $\tilde w \in W^{c}$; such formulas will be provided in Section \ref{sec:met_results}.  Our formulas are in terms of (nonmetaplectic) \textit{parahoric} Whittaker functions, defined and studied in Section 4 of \cite{BBBG19}.  We briefly recall the relevant background here.  Let $J \subseteq [1,r]$ with associated parabolic subgroup $W_{J}$.   Let $w \in W^{J}$.  Then the parahoric Whittaker function $\psi_{w}^{J}$ is defined as in equation (18) of \cite{BBBG19} as a sum of Iwahori-Whittaker functions over the subgroup $W_{J}$.  

We will need the following result from \cite{BBBG19}, which gives a formula for $\psi_{w}^{J}$ in terms of the action of the Demazure-Whittaker operators $\mathcal{T}_i$ (recall Definition \ref{def:DW_opers}).

\begin{theorem}\label{thm:parahoric_to_T} Let $J \subseteq [1,r]$ and $w \in W^{J}$.  Let $w' \in W$ and $\lambda + \rho \in \Lambda^{+}$.
Then
\begin{equation*}
\psi_{w}^{J}(\mathbf{y}; \varpi^{-\lambda}w'; v) = v^{l(w')} \cdot \mathbf{y}^{-\rho} \sum_{u \in W_{J}} \mathcal{T}_{w u, v^{-1/2}} \mathcal{T}_{w', v^{-1/2}}^{-1} \mathbf{y}^{\lambda + \rho}.
\end{equation*}
\end{theorem}
\begin{proof}
Follows from \cite{BBBG19} equation (18) and \cite{BBBG19} Corollary 3.9, using the relation (29) of \cite{BBBG19} to translate between their operators $\mathfrak{T}_{i}$ and $\mathcal{T}_{i}$, and the fact that $\mathcal{T}_{wu} = \mathcal{T}_{w} \mathcal{T}_{u}$ for $u \in W_{J}$.
\end{proof}

\section{Matrix coefficients of (anti-)symmetrizers}\label{sec:AHA_results}

Let $h \in \widetilde{\mathcal{H}}_{\Lambda}$.  In this section we calculate the matrix coefficients $\gamma_{w}(\mathbf{1}^{\pm} h)$ for left multiplication by the (anti-) symmetrizer (recall Definitions \ref{def:gamma_coeffs_Hecke}, \ref{def:symm_antisymm}).  Our formulas are in terms of the polynomial representation $\pi$.  

\begin{definition}
Define the $\mathbb{F}$-linear operators $A_{w, \hat{w}}^{\pm} : \mathbb{F}[\Lambda] \rightarrow \mathbb{F}[\Lambda]$ for $w, \hat{w} \in W$ by
\begin{equation*}
A_{w, \hat{w}}^{\pm} (f) := \gamma_{w}( \mathbf{1}^{\pm} f T_{\hat w})
\end{equation*}
for $f \in \mathbb{F}[\Lambda]$.
\end{definition}

For any $h \in \widetilde{\mathcal{H}}_{\Lambda}$ and $w \in W$, in order to determine $\gamma_{w}(\mathbf{1}^{\pm} h)$, it suffices to compute $A_{w, \hat{w}}^{\pm}(\gamma_{\hat w} (h))$ for all $\hat{w} \in W$.  Indeed, by linearity, we have
\begin{equation}\label{eq:lt_mult_symm}
\gamma_{w}(\mathbf{1}^{\pm} h) = \gamma_{w}\Big(\mathbf{1}^{\pm} \sum_{\hat{w} \in W} \gamma_{\hat w}(h) T_{\hat w}\Big) = \sum_{\hat w \in W} \gamma_{w}(\mathbf{1}^{\pm}  \gamma_{\hat w}(h) T_{\hat w}))
= \sum_{\hat{w} \in W} A_{w, \hat{w}}^{\pm}(\gamma_{\hat w} (h)).
\end{equation}

It will be convenient to work with the following reparameterizations of the $A$-operators.

\begin{definition}\label{B_opers}
Define the $\mathbb{F}$-linear operators $B_{w, \hat{w}}^{\pm} : \mathbb{F}[\Lambda] \rightarrow \mathbb{F}[\Lambda]$ for $w, \hat{w} \in W$ by 
\begin{align*}
B_{w, \hat{w}}^{+} &= w_0 A^{+}_{w_0 w, \hat{w}} \\
B_{w, \hat{w}}^{-} &= \iota w_0 A_{w_0 w, \hat{w}}^{-} \iota. 
\end{align*}
\end{definition}

\begin{lemma}\label{Bbase}
We have
\begin{align*}
B_{1,1}^{\pm} &= (\pm 1)^{ l(w_0)} t(w_0)^{\pm 1}. 
\end{align*}
\end{lemma}
\begin{proof}
Let $f \in \mathbb{F}[\Lambda]$.  Then we have
\begin{equation*}
A^{\pm}_{w_0, 1}(f) = \gamma_{w_0}(\mathbf{1}^{\pm} f) = \gamma_{w_0}\big(\sum_{w \in W} (\pm 1)^{l(w)} t(w)^{\pm 1} T_w f \big) = (\pm 1)^{ l(w_0)} t(w_0)^{\pm 1} w_0 f,
\end{equation*}
where we have applied Lemma \ref{lem:T_triang}. The result then follows from Definition \ref{B_opers}.
\end{proof}

\begin{lemma}\label{B_wrecur}
Let $w, \hat{w} \in W$ and $1 \leq i \leq r$.  If $l(s_i w) > l(w)$ then we have
 \begin{align*}
 B^{+}_{s_i w, \hat{w}} &= \pi(T_i^{-1}) B^{+}_{w, \hat{w}} \\
 B^{-}_{s_i w, \hat{w}} &= - \pi(T_i)B_{w, \hat{w}}^{-}.
 \end{align*}
 
\end{lemma}
\begin{proof}
Let $j = w_0(i)$, and note that $l(s_j w_0 w) = l(ws_iw) < l(w_0w)$. Let $f \in \mathbb{F}[\Lambda]$.  

By \eqref{T1}, we have 
\begin{align*}
\gamma_{w_0w}(T_j \mathbf{1}^{\pm} f T_{\hat w}) &= (\pm t_i)^{\pm 1} \gamma_{w_0w}(\mathbf{1}^{\pm} f T_{\hat w}) = (\pm t_i)^{\pm 1} A_{w_0w, \hat{w}}^{\pm}(f) \end{align*}
On the other hand, by Lemma \ref{gamma_Th}, the LHS is equal to
\begin{multline*}
(t_i - t_i^{-1}) \nabla_{j} \gamma_{w_0w}(\mathbf{1}^{\pm} f T_{\hat w}) + s_j \gamma_{s_j w_0w}(\mathbf{1}^{\pm} f T_{\hat w}) + (t_i - t_i^{-1}) s_j \gamma_{w_0w}(\mathbf{1}^{\pm} f T_{\hat w})\\
= (t_i - t_i^{-1}) \nabla_j A_{w_0w,\hat w}^{\pm}(f) + s_j A_{w_0s_iw, \hat{w}}^{\pm}(f) + (t_i - t_i^{-1}) s_j A^{\pm}_{w_0w, \hat w}(f).
\end{multline*}
Setting these equal, multiplying both sides by $s_j$ and rearranging terms, we obtain
\begin{align*}
A_{w_0s_i w, \hat{w}}^{\pm}(f) &=  (\pm t_i)^{\pm 1} s_j A_{w_0w, \hat{w}}^{\pm}(f) - (t_i - t_i^{-1}) s_j \nabla_j A_{w_0w, \hat{w}}^{\pm}(f) - (t_i - t_i^{-1}) A^{\pm}_{w_0w, \hat w}(f)\\
&= (\pm t_i)^{\mp 1} s_j A_{w_0w, \hat{w}}^{\pm}(f) - (t_i - t_i^{-1}) \nabla_j A_{w_0w, \hat{w}}^{\pm}(f),
\end{align*}
where we have applied Lemma \ref{lem:sinablai}. We can rewrite these as follows:
\begin{align*}
A_{w_0s_iw, \hat{w}}^{+}(f) &= -\pi(T_j) A_{w_0s_iw, \hat{w}}(f) + (t_i + t_i^{-1}) s_j A_{w_0w, \hat{w}}^{+}(f) \\
A_{w_0s_iw, \hat{w}}^{-}(f) &= -\pi(T_j) A_{w_0s_iw, \hat{w}}(f).
\end{align*}	
The results now follow from Definition \ref{B_opers} and Lemma \ref{lem:w0Ti}.
\end{proof}

\begin{lemma}\label{B_whrecur}
Let $w, \hat{w} \in W$ and $1 \leq i \leq r$. If $l(s_i \hat{w}) > l(\hat w)$, we have
\begin{align*}
B^{+}_{w, s_i \hat{w}} &= B_{w, \hat{w}}^{+} \pi(T_i) \\
B^{-}_{w, s_i \hat{w}} &= - B_{w, \hat{w}}^{-} \pi(T_i^{-1}).
\end{align*}
\end{lemma}

\begin{proof}
Let $f \in \mathbb{F}[\Lambda]$.  By \eqref{T2}, we have
\begin{equation*}
\gamma_{w}(\mathbf{1}^{\pm} T_i f T_{\hat w}) = (\pm t_i)^{\pm 1} \gamma_{w}(\mathbf{1}^{\pm} f T_{\hat w}) =(\pm t_i)^{\pm 1} A_{w,\hat{w}}^{\pm}(f).
\end{equation*}
On the other hand, by \eqref{Tf_commut} and \eqref{TT}, the LHS is equal to
\begin{equation*}
\gamma_{w}\big(\mathbf{1}^{\pm}(s_i f T_{s_i \hat{w}} + (t_i - t_i^{-1}) \nabla_{i} f T_{\hat{w}})\big) = A_{w, s_i \hat{w}}^{\pm}(s_i f) + (t_i - t_i^{-1}) A_{w, \hat{w}}^{\pm}(\nabla_i f)
\end{equation*}

Consider first the $A^{+}$ case.  Setting the two expressions equal and putting $f := s_ig$, we obtain
\begin{align*}
A_{w, s_i \hat{w}}^{+}(g) &= A_{w, \hat{w}}^{+}\Big( t_i s_i g - (t_i - t_i^{-1}) \nabla_i (s_i g)\Big).
\end{align*}
Using $\nabla_{i} s_i g = -\nabla_i g$, along with Theorem \ref{thm:pol_rep}, the RHS is equal to $A_{w, \hat{w}}^{+}\big( \pi(T_i) g \big)$ as desired.

Now consider the $A^{-}$ case.  Setting the two expressions equal and putting $f :=  s_i w_0 g$ we obtain
\begin{align*}
A_{w, s_i \hat{w}}^{-}(w_0 g) &= A_{w, \hat{w}}^{-}\Big( -t_i^{-1} s_i w_0 g - (t_i - t_i^{-1}) \nabla_i (s_i w_0 g)\Big).
\end{align*}
Let $1 \leq j \leq r$ be such that $w_0(j) = i$.  The RHS above is equal to 
\begin{align*}
A_{w, \hat{w}}^{-}\Big( -t_i^{-1} w_0 s_j g - (t_i - t_i^{-1}) \nabla_i (w_0 s_j g)\Big).
\end{align*}
By Lemma \ref{w0sinabla}, this is equal to 
\begin{align*}
A_{w, \hat{w}}^{-}\Big( -t_i^{-1} w_0 s_j g + (t_i - t_i^{-1}) w_0 s_j \nabla_{j}(s_j g) \Big).
\end{align*}
By \eqref{pi_Tinvf} and using $\nabla_{j} s_j g = - \nabla_{j} g$, the RHS is equal to $-A_{w, \hat{w}}^{-}(w_0 \pi(T_{w_0(i)}^{-1})(g))$.
\end{proof}

The following result, which is the main theorem of this section, expresses $A^{\pm}_{w, \hat{w}}$  in terms of the operators $\pi(T_{u})$ of Theorem \ref{thm:pol_rep} (for $w, \hat{w}, u \in W$).

\begin{theorem}\label{thm:main_thm_HA}
Let $w, \hat{w} \in W$.  We have
\begin{align*}
A^{+}_{w, \hat{w}} &= t(w_0)  w_0 \pi(T_{(w_0w)^{-1}}^{-1} T_{\hat{w}^{-1}}) \\
A^{-}_{w, \hat{w}} &= t(w_0)^{-1} (-1)^{l( w) + l(\hat{w})} \iota w_0 \pi(T_{w_0 w} T_{\hat{w}}^{-1}) \iota.
\end{align*}
\end{theorem}
\begin{proof}
Applying Lemmas \ref{Bbase}, \ref{B_wrecur}, \ref{B_whrecur}, we have
\begin{equation}
	\begin{split}
	B^{+}_{w, \hat{w}} &=t(w_0)  \pi(T_{w^{-1}}^{-1} T_{\hat{w}^{-1}}) \\
	B^{-}_{w, \hat{w}} &= (-1)^{l(w_0)} t(w_0)^{-1} (-1)^{l( w) + l(\hat{w})} \pi(T_w T_{\hat{w}}^{-1}).
	\end{split}
\end{equation}	
The result then follows from Definition \ref{B_opers}.
\end{proof}

Restating Theorem \ref{thm:main_thm_HA} in terms of the $\gamma$-coefficient functions yields the following corollary.

\begin{corollary}\label{cor:1fT}
Let $w, \hat{w} \in W$ and $f \in \mathbb{F}[\Lambda]$.  We have
\begin{align*}
\gamma_{w}(\mathbf{1}^{+} f T_{\hat{w}}) &= t(w_0) w_0 \pi(T_{(w_0 w)^{-1}}^{-1}) \pi(T_{\hat{w}^{-1}})f \\
\gamma_{w}(\mathbf{1}^{-} f T_{\hat{w}}) &=t(w_0)^{-1} (-1)^{l(w)} w_0 \iota \pi(T_{w_0 w}) (-1)^{l(\hat w)} \pi(T_{\hat w}^{-1}) \iota f.
\end{align*}

\end{corollary}

Using \eqref{eq:lt_mult_symm} and Theorem \ref{thm:main_thm_HA}, we obtain the following decompositions for $\mathbf{1}^{\pm} h$ for $h \in \widetilde{\mathcal{H}}_{\Lambda}$.

\begin{corollary}\label{cor:PBW1pm}
Let $h \in \widetilde{\mathcal{H}}_{\Lambda}$.  We have
\begin{align*}
\mathbf{1}^{+} h &= \sum_{w \in W}  t(w_0) w_0 \Big( \sum_{\hat w \in W} \pi( T_{(w_0 w)^{-1}}^{-1} T_{\hat{w}^{-1}}) \gamma_{\hat w}(h) \Big)T_{w}  \\
\mathbf{1}^{-} h &= \sum_{w \in W} t(w_0)^{-1} (-1)^{l(w)} w_0 \iota \Big( \sum_{\hat w \in W}  (-1)^{l(\hat w)} \pi(T_{w_0 w} T_{\hat w}^{-1}) \iota \gamma_{\hat w}(h) \Big) T_{w}.
\end{align*}
\end{corollary}

\begin{remark}
Suppose $h \in \mathbb{F}[\Lambda]^{W} \subseteq \widetilde{\mathcal{H}}_{\Lambda}$.  In this case, $\gamma_{1}(h) = h$ and $\gamma_{\hat{w}}(h) = 0$ for $\hat{w} \neq 1$.  So, by Corollary \ref{cor:PBW1pm},
\begin{align*}
\mathbf{1}^{+} h = \sum_{w \in W} t(w_0) w_0 \big( \pi(T_{(w_0 w)^{-1}}^{-1}) h \big) T_{w} \\
\mathbf{1}^{-} h = \sum_{w \in W} t(w_0)^{-1} (-1)^{l(w)} w_0 \iota \big( \pi(T_{w_0 w}) \iota h \big) T_{w}.
\end{align*}
Using $\pi(T_{w}) h = t(w) h$ for all $w \in W$ and Lemma \ref{lem:w0Ti}, we have
\begin{align*}
\mathbf{1}^{+} h = \sum_{w \in W} t(w_0) t(w_0 w)^{-1} h T_{w} \\
\mathbf{1}^{-} h = \sum_{w \in W} t(w_0)^{-1} (-1)^{l(w)} t(w w_0) h  T_{w}.
\end{align*}
Finally, using $t(w_0 w)  = t(w w_0) = t(w_0) t(w)^{-1}$, we recover the well-known fact that $h$ commutes with $\mathbf{1}^{\pm}$.

\end{remark}

\section{(Anti-)symmetric quasi-polynomial duality}\label{sec:QP_results}

We will now use the results of Section \ref{sec:AHA_results} to give explicit formulas for certain classes of quasi-polynomials.  Following Section \ref{sec:qp}, assume $\Lambda \in \mathcal{L}$ and $c \in C^{0}_{\Lambda}$.    Let $\hat{w}, w \in W^{c}$ and $f \in \mathbb{F}[\Lambda]$ so $\mathbf{1}^{\pm} f T_{\hat{w}} \in \widetilde{\mathcal{H}}_{\Lambda}$.  In this section, we provide explicit decomposition formulas for the quasi-polynomials $ \pi^{qp}(\mathbf{1}^{\pm} f T_{\hat{w}}) x^{c}$, by computing the coefficient functions $\gamma_{w}^{qp}\big( \pi^{qp}(\mathbf{1}^{\pm} f T_{\hat{w}}) x^{c}\big) $ (recall Definition \ref{def:gamma_coeff_qp}).  By Corollary \ref{cor:gammas}, we have 
\begin{equation}\label{qp_coeff}
\gamma_{w}^{qp}\big( \pi^{qp}(\mathbf{1}^{\pm} f T_{\hat{w}}) x^{c}\big)  = \sum_{u \in W_{c}} t(u) \gamma_{wu}(\mathbf{1}^{\pm} f T_{\hat{w}}). 
\end{equation}
We will use the results of Section \ref{sec:AHA_results} to compute the RHS of \eqref{qp_coeff}, handling the $\mathbf{1}^{+}$ and $\mathbf{1}^{-}$ cases separately.

Note that, by Theorem \ref{thm:XT_formula} and Proposition \ref{prop:ht},  in the particular case above when $f = x^{\mu}$, for $\mu \in \Lambda$, we have
\begin{equation} \label{eq:P-pols}
\pi^{qp}(\mathbf{1}^{\pm} x^{\mu} T_{\hat{w}}) x^{c} = \pi^{qp}(\mathbf{1}^{\pm}) x^{\mu + \hat{w}c} =: p_{\mu + \hat{w} c}^{\pm}
\end{equation}
As discussed in the introduction and in Section \ref{sec:connections_to_pols}, these objects are the $q \rightarrow \infty$ limit of quasi-polynomial generalizations of (anti-)symmetric Macdonald polynomials.

\begin{theorem}\label{thm:qp_1pfT}
Let $\hat{w}, w \in W^{c}$ and $f \in \mathbb{F}[\Lambda]$.  We have
\begin{equation*}
\gamma_{w}^{qp}\big( \pi^{qp}(\mathbf{1}^{+} f T_{\hat{w}}) x^{c}\big)  = t(w_0) w_0 \pi \Big( T_{(w_0 w)^{-1}}^{-1} \mathbf{1}^{+}_{W_{c}} T_{\hat{w}^{-1}}\Big) f.
\end{equation*}
\end{theorem}
\begin{proof}

By Corollary \ref{cor:1fT}, the RHS of \eqref{qp_coeff} is equal to
\begin{equation*}
\sum_{u \in W_{c}} t(u) t(w_0) w_0 \pi( T_{(w_0 w u)^{-1}}^{-1} T_{\hat{w}^{-1}} ) f.
\end{equation*}
Using Lemma \ref{lem:Tw0wu}, this is equal to
\begin{equation*}
t(w_0) w_0 \pi \big( T^{-1}_{{(w_0w)}^{-1}}  \big) \sum_{u \in W_{c}} t(u) \pi(T_u) \pi(T_{\hat{w}^{-1}}) f.
\end{equation*}
Using Lemma \ref{lem:Tw0} and rewriting in terms of the partial symmetrizer $\mathbf{1}_{W_{c}}^{+}$ gives the result.  \end{proof}

\begin{corollary}\label{cor:symm_char}
Let $w \in W^{c}$ and $h \in \widetilde{\mathcal{H}}_{\Lambda}$.  We have
\begin{equation*}
\gamma_{w}^{qp}\big( \pi^{qp}(\mathbf{1}^{+} h) x^{c}\big)  = t(w_0) w_0 \pi \Big( T_{(w_0 w)^{-1}}^{-1} \mathbf{1}^{+}_{W_{c}} \Big)  \sum_{\hat{w} \in W^{c}} \pi(T_{\hat{w}^{-1}}) \Big( \sum_{u \in W_{c}} t(u) \gamma_{\hat{w} u}(h) \Big).
\end{equation*}
Consequently,
\begin{equation*}
\pi^{qp}(\mathbf{1}^{+} h) x^{c} = t(w_0) w_0 \sum_{w \in W^{c}} \pi \Big( T_{(w_0 w)^{-1}}^{-1} \mathbf{1}^{+}_{W_{c}} \Big)  \sum_{\hat{w} \in W^{c}} \pi(T_{\hat{w}^{-1}}) \Big( \sum_{u \in W_{c}} t(u) \gamma_{\hat{w} u}(h) \Big) x^{w c}.
\end{equation*}

\end{corollary}

\begin{proof}
We have
\begin{equation*}
h = \sum_{w \in W} \gamma_{w}(h) T_{w} = \sum_{v \in W^{c}} \sum_{u \in W_{c}} \gamma_{vu}(h) T_{v}T_{u}.
\end{equation*}
Using Theorem \ref{thm:XT_formula} and Proposition \ref{prop:ht}, we have
\begin{equation*}
\pi^{qp}(\mathbf{1}^{+} h) x^{c} = \pi^{qp}\Big(\mathbf{1}^{+} \sum_{v \in W^{c}} \Big( \sum_{u \in W_{c}} t(u) \gamma_{vu}(h)\Big) T_{v} \Big) x^{c}.
\end{equation*}
The result now follows by linearity from Theorem \ref{thm:qp_1pfT}.

\end{proof}

\begin{theorem}\label{thm:qp_1mfT}
Let $\hat{w}, w \in W^{c}$ and $f \in \mathbb{F}[\Lambda]$.  We have
\begin{equation*}
\gamma_{w}^{qp}\big( \pi^{qp}(\mathbf{1}^{-} f T_{\hat{w}}) x^{c}\big)  = t(w_0)^{-1} t(w_0(W_{c}))^{2} (-1)^{l(\hat{w}) + l(w)} \iota w_0 \pi \Big( T_{w_0 w} \mathbf{1}^{-}_{W_{c}} T_{\hat{w}}^{-1} \Big) \iota f.
\end{equation*}
\end{theorem}
\begin{proof}
By Corollary \ref{cor:1fT}, the RHS of \eqref{qp_coeff} is equal to
\begin{equation*}
\sum_{u \in W_{c}} t(u) t(w_0)^{-1} (-1)^{l(wu) + l(\hat{w})} \iota w_0 \pi(T_{w_0 w u} T_{\hat{w}}^{-1}) \iota f.
\end{equation*}
Using Lemma \ref{lem:Tw0wu} and Lemma \ref{lem:lensplit}, this is equal to
\begin{equation*}
t(w_0)^{-1} (-1)^{l(\hat{w}) + l(w)} \iota w_0 \pi \Big( T_{w_0 w} \big( \sum_{u \in W_{c}} t(u) (-1)^{l(u)} T_{u^{-1}}^{-1} \big) T_{\hat{w}}^{-1} \Big) \iota f.
\end{equation*}
Since the summand inside the parantheses is equal to $t(w_0(W_{c}))^{2} \mathbf{1}^{-}_{W_{c}}$ by Proposition \ref{prop:ptl1minus}, we obtain the result.
\end{proof}

Analogous to Corollary \ref{cor:symm_char}, we have the following result.

\begin{corollary}\label{cor:antisymm_char}
Let $w \in W^{c}$ and $h \in \widetilde{\mathcal{H}}_{\Lambda}$.  We have
\begin{multline*}
\gamma_{w}^{qp}\big( \pi^{qp}(\mathbf{1}^{-} h) x^{c}\big)  = t(w_0)^{-1} t(w_0(W_{c}))^{2} \cdot \\
\cdot \sum_{\hat{w} \in W^{c}} (-1)^{l(\hat{w}) + l(w)} \iota w_0 \pi \Big( T_{w_0 w} \mathbf{1}^{-}_{W_{c}} T_{\hat{w}}^{-1} \Big) \iota \Big( \sum_{u \in W_{c}} t(u) \gamma_{\hat{w} u}(h) \Big).
\end{multline*}
Consequently,
\begin{multline*}
\pi^{qp}(\mathbf{1}^{-} h) x^{c}  =  t(w_0)^{-1} t(w_0(W_{c}))^{2} \cdot \\ 
\cdot \sum_{w \in W^{c}} (-1)^{ l(w)} \iota w_0 \pi \Big( T_{w_0 w} \mathbf{1}^{-}_{W_{c}}\Big)  \sum_{\hat{w} \in W^{c}}  (-1)^{l(\hat{w})} \pi(T_{\hat{w}}^{-1} ) \iota \Big( \sum_{u \in W_{c}} t(u) \gamma_{\hat{w} u}(h) \Big) x^{w c}.
\end{multline*}
\end{corollary}

From \eqref{eq:P-pols} and Theorems \ref{thm:qp_1pfT} and \ref{thm:qp_1mfT}, as well as Corollary \ref{cor:qp1y_domin} and Definition \ref{def:MDptl}, we obtain the following result.

\begin{corollary}\label{cor:qp_main} 
Let $c \in C^{0}_{\Lambda}$, $y \in \widetilde{\mathcal{O}}_{c}$ and write (uniquely) $y = \mu + \hat{w} c$, where $\mu \in \Lambda$ and $\hat{w} \in W^{c}$.  Then we have the following formulas for the quasi-polynomials $p^{\pm}_{y}$:
\begin{align*}
p^{+}_{y} &= t(w_0)  \sum_{w \in W^{c}}\Bigg[ w_0 \pi \Big( T_{(w_0 w)^{-1}}^{-1} \mathbf{1}^{+}_{W_{c}} T_{\hat{w}^{-1}}\Big) x^{\mu} \Bigg] x^{w c}\\
p^{-}_{y} &= \frac{(-1)^{l(\hat{w})} t(w_0(W_{c}))^{2}}{t(w_0) } \sum_{w \in W^{c}}  (-1)^{ l(w)} \Bigg[ \iota w_0 \pi \Big( T_{w_0 w} \mathbf{1}^{-}_{W_{c}} T_{\hat{w}}^{-1} \Big) \iota x^{\mu} \Bigg] x^{wc}.
\end{align*}
Moreover, if $y$ is dominant, we have
\begin{align*}
\gamma_{w}( \iota \overline{E}_{-y}^{+}(q; \mathbf{t}^{-1}; \tau)) &=  t(w_0)^{-1}  \cdot w_0 \pi (T_{(w_0 w)^{-1}}^{-1}) p_{\hat{w}^{-1} \mu}^{J_c, +} \\
\gamma_{w}( \iota \overline{E}_{-y}^{-}(q; \mathbf{t}^{-1}; \tau)) &= (-1)^{l(\hat{w}) + l(w)} t(w_0(W_c))^{2} t(w_0) \cdot   \iota w_0 \pi ( T_{w_0 w}) p_{\hat{w}^{-1} \mu}^{J_c, -}.
\end{align*}
\end{corollary}

\begin{remark}
The specialization of Corollary \ref{cor:qp_main} to the $GL_{r}$ context gives Theorem \ref{cor:qp_main_GL}.
\end{remark}

\begin{remark} 
Let $w_0^{c} \in W^{c}$ and ${w_0}_{c} \in W_{c}$ as in Definition \ref{def:wc_decomp}.  Then, since $w_0 {w_0}^{c} = {w_0}_{c}^{-1}$, by Theorem \ref{cor:qp_main} we have the following expression for the $\gamma_{w_0^{c}}$-coefficient of $p^{\pm}_{\mu + \hat{w} c}$:
\begin{align*}
\gamma_{w_0^{c}}(p^{+}_{\mu + \hat{w} c}) &= \frac{ t(w_0) }{t({w_0}_c) } w_0 \pi( \mathbf{1}^{+}_{W_{c}} T_{\hat{w}^{-1}}) x^{\mu} \\
\gamma_{w_0^{c}}(p^{-}_{\mu + \hat{w} c} ) &= \frac{(-1)^{l(\hat{w}) + l(w_0)} t(w_0(W_{c}))^{2}}{t({w_0}_c) t(w_0) } \iota w_0 \pi( \mathbf{1}^{-}_{W_{c}} T_{\hat{w}}^{-1}) x^{-\mu}
\end{align*}
and the relation between the $w$ and $w_0^{c}$-coefficients is given by
\begin{align*}
\gamma_{w}(p^{+}_{\mu + \hat{w} c}) &= t({w_0}_{c}) w_0 \pi(T_{(w_0 w)^{-1}}^{-1})  w_0 \gamma_{w_0^{c}}(p^{+}_{\mu + \hat{w} c})  \\
\gamma_{w}(p^{-}_{\mu + \hat{w} c} ) &=(-1)^{l(w_0) + l(w)} t({w_0}_{c}) \iota w_0 \pi(T_{w_0 w}) \iota w_0 \gamma_{w_0^{c}}(p^{-}_{\mu + \hat{w} c})
\end{align*}
for $w \in W^{c}$.
\end{remark}

\begin{remark}Special cases of Theorem \ref{cor:qp_main}.
\begin{enumerate}
\item Suppose $W_{c} = W$, so $W^{c} = \{ 1\}$ and $c = 0$, $\hat{w} = 1$.  We obtain
\begin{align*}
p^{+}_{\mu + \hat{w} c} &=    w_0 \pi(\mathbf{1}^{+} ) x^{\mu} = \pi(\mathbf{1}^{+} ) x^{\mu} \\
p^{-}_{\mu + \hat{w} c} &=   \frac{(-1)^{l(\hat{w})} t(w_0(W_{c}))^{2}}{t(w_0) } \iota w_0 \pi(\mathbf{1}^{-} ) x^{-\mu} =  \frac{(-1)^{l(\hat{w})} t(w_0(W_{c}))^{2}}{t(w_0) } \pi(\mathbf{1}^{-} ) x^{w_0\mu}
\end{align*}

\item Suppose $W_{c} = \{ 1 \}$, so $W^{c} = W$.  We obtain
\begin{align*}
p^{+}_{\mu + \hat{w} c} &=  t(w_0)  \sum_{w \in W}\Bigg[ w_0 \pi \Big( T_{(w_0 w)^{-1}}^{-1} T_{\hat{w}^{-1}}\Big) x^{\mu} \Bigg] x^{w c}\\
p^{-}_{\mu + \hat{w} c} &= (-1)^{l(\hat{w})} t(w_0)\sum_{w \in W}  (-1)^{ l(w)} \Bigg[ \iota w_0 \pi \Big( T_{w_0 w} T_{\hat{w}}^{-1} \Big) \iota x^{\mu} \Bigg] x^{wc}.
\end{align*}

\end{enumerate}

\end{remark}

%%%%%%%%%%%%%%%%%%%%%%%%%%%%%%%%

\section{Parahoric-Metaplectic duality}\label{sec:met_results}

We will now use the results of Section \ref{sec:AHA_results} to give explicit formulas for certain classes of metaplectic polynomials.  Assume the setup and notation of Section \ref{sec:metreps}, and that $\Lambda \in \mathcal{L}$ and $c \in C^{0}_{\Lambda^{m}}$.  Let $\hat{w}, w \in W^{c}$ and $f \in \mathbb{F}[\Lambda^{m}]$.   We will provide explicit decomposition formulas for the metaplectic polynomials $ \pi^{m}_{v}(\mathbf{1}^{\pm} f T_{\hat{w}}) x^{c}$, by computing the coefficient functions $\gamma_{w}^{m}\big( \pi^{m}_{v}(\mathbf{1}^{\pm} f T_{\hat{w}}) x^{c }\big)$ (recall Definition \ref{def:gamma_coeff_met}).

Recall the statistics $h^{m,c}_{v}$ from Definition \ref{def:met_h_stats} (under the equal Hecke and representation parameter specialization). Note that, by Theorem \ref{thm:XT_met} and Proposition \ref{prop:ht_met},  in the particular case when $f = x^{\mu}$, for $\mu \in \Lambda^{m}$, we have
\begin{equation*}
\pi^{m}_{v}(\mathbf{1}^{\pm} x^{\mu} T_{\hat{w}}) x^{c} = h^{m,c}_{v}(\hat{w}) \pi^{m}_{v}(\mathbf{1}^{\pm}) x^{\mu + \hat{w}c}.
\end{equation*}

By Corollary \ref{met_coeffs}, we have 
\begin{equation}\label{eq:met_coeff}
\gamma_{w}^{m}\big( \pi^{m}_{v}(\mathbf{1}^{\pm} f T_{\hat{w}}) x^{c }\big)  = h^{m,c}_{v}(w) \sum_{u \in W_{c} } v^{l(u)/2} \gamma_{wu}(\mathbf{1}^{\pm} f T_{\hat{w}}). 
\end{equation}
We will use the results of Section \ref{sec:AHA_results} to compute the RHS of \eqref{eq:met_coeff}, handling the $\mathbf{1}^{+}$ and $\mathbf{1}^{-}$ cases separately.  Note that the arguments are analogous to Section \ref{sec:QP_results}, the only difference being that in this section we use the results of Section \ref{sec:AHA_results} for the \textit{metaplectic} Hecke algebra $ \widetilde{\mathcal{H}}_{\Lambda^{m}}$.

For ease of notation, we first define the following coefficients.  Let $\hat{w}, w \in W$ and let
\begin{align*}
F_{\mathbf{1}^{+}}^{c}(w) &= h^{m,c}_{v}(w) v^{l(w_0)/2} \\
F_{\mathbf{1}^{-}}^{c}(w) &= h^{m, c}_{v}(w) v^{-l(w_0)/2 + l(w_0(W_c))} (-1)^{ l(w)}. 
\end{align*}

\begin{theorem}\label{thm:met_plus_main}
Let $\hat{w}, w \in W^{c}$ and $f \in \mathbb{F}[\Lambda^{m}]$.  We have
\begin{equation*}
\gamma_{w}^{m}\big( \pi^{m}_{v}(\mathbf{1}^{+} f T_{\hat{w}}) x^{c}\big)  = F_{\mathbf{1}^{+}}^{c}(w) \cdot w_0 \pi^{m}_{v} \Big( T_{(w_0 w)^{-1}}^{-1} \mathbf{1}^{+}_{W_{c}} T_{\hat{w}^{-1}}\Big) f.
\end{equation*}
\end{theorem}
\begin{proof}

By Theorem \ref{cor:1fT}, the RHS of \eqref{eq:met_coeff} is equal to
\begin{equation*}
h^{m, c}_{v}(w) \sum_{u \in W_{c}} v^{l(u)/2 + l(w_0)/2} w_0 \pi^{m}_{v}( T_{(w_0 w u)^{-1}}^{-1} T_{\hat{w}^{-1}} ) f.
\end{equation*}
Using Lemma \ref{lem:Tw0wu}, this is equal to
\begin{equation*}
F_{\mathbf{1}^{+}}(w) w_0 \pi^{m}_{v} \big( T^{-1}_{{(w_0w)}^{-1}}  \big) \sum_{u \in W_{c}} v^{l(u)/2} \pi^{m}_{v}(T_u) \pi^{m}_{v}(T_{\hat{w}^{-1}}) f.
\end{equation*}
Using Lemma \ref{lem:Tw0} and rewriting in terms of the partial symmetrizer $\mathbf{1}_{W_{c}}^{+}$ gives the result.
\end{proof}

\begin{theorem}\label{thm:met_minus_main}
Let $\hat{w}, w \in W^{c}$ and $f \in \mathbb{F}[\Lambda^{m}]$.  We have
\begin{equation*}
\gamma_{w}^{m}\big( \pi^{m}_{v}(\mathbf{1}^{-} f T_{\hat{w}}) x^{c}\big)  = (-1)^{l(\hat{w})} F_{\mathbf{1}^{-}}^{c}(w) \cdot \iota w_0 \pi^{m}_{v} \Big( T_{w_0 w} \mathbf{1}^{-}_{W_{c}} T_{\hat{w}}^{-1} \Big) \iota f.
\end{equation*}
\end{theorem}
\begin{proof}
By Theorem \ref{cor:1fT}, the RHS of \eqref{eq:met_coeff} is equal to
\begin{equation*}
h^{m,c}_{v}(w) \sum_{u \in W_{c}} v^{l(u)/2 - l(w_0)/2} (-1)^{l(wu) + l(\hat{w})} \iota w_0 \pi^{m}_{v}(T_{w_0 w u} T_{\hat{w}}^{-1}) \iota f.
\end{equation*}
Using Lemma \ref{lem:Tw0wu} and Lemma \ref{lem:lensplit}, this is equal to
\begin{equation*}
h^{m, c}_{v}(w) v^{-l(w_0)/2} (-1)^{l(\hat{w}) + l(w)} \iota w_0 \pi^{m}_{v} \Big( T_{w_0 w} \big( \sum_{u \in W_{c}}  (-v^{1/2})^{l(u)} T_{u^{-1}}^{-1} \big) T_{\hat{w}}^{-1} \Big) \iota f.
\end{equation*}
Since the summand inside the parantheses is equal to $v^{l(w_0(W_{c} ))} \mathbf{1}^{-}_{W_{c}}$ by Proposition \ref{prop:ptl1minus}, we obtain the result.
\end{proof}

\begin{corollary}\label{cor:coeffs_w0}
Let $\mu \in \Lambda$ with decomposition $w_0 \mu = \eta + \hat{w} c$ for some $\eta \in \Lambda^{m}$, $\hat{w} \in W^{c}$, and $c \in C^{0}_{\Lambda^m}$.  Let $\tilde w \in W$.  Then
\begin{multline*}
w_0 \iota \gamma_{w_0 \tilde w^{c}  w_0 , -w_0c}^{m}\big( \pi^{m}_{v}(\mathbf{1}^{-}) x^{- \mu}\big) \\= \frac{(-1)^{l( w_0 \hat{w} w_0)} F_{\mathbf{1}^{-}}^{-w_0 c}( w_0 \tilde w^{c} w_0) }{h^{m, -w_0 c}_{v}(w_0 \hat{w} w_0 )}  \cdot \pi^{m}_{v} \Big( T_{ \tilde w^{c} w_0 } \mathbf{1}^{-}_{W_{-w_0 c}} T_{w_0 \hat{w} w_0}^{-1} \Big)  x^{w_0 \eta }. 
\end{multline*}
\end{corollary}
\begin{proof}
First, note that we have the decomposition $-\mu = -w_0 \eta + w_0 \hat{w} w_0 (-w_0 c)$ and $-w_0 \nu \in \Lambda^{m}$, $-w_0 c \in C^{0}_{\Lambda^m}$.  Also note that $w_0 \tilde w^{c} w_0 \in W^{-w_0 c}$ by Lemma \ref{lem:w0_conj}.

By Theorem \ref{thm:XT_met}, we have
\begin{equation*}
\pi^{m}_{v}(\mathbf{1}^{-}) x^{-\mu} = \pi^{m}_{v}(\mathbf{1}^{-} x^{-w_0\eta}) x^{ w_0 \hat{w} w_0 (-w_0 c)} = \frac{1}{h^{m, -w_0 c}_{v}(w_0 \hat{w} w_0 )} \pi^{m}_{v}(\mathbf{1}^{-} x^{-w_0 \eta} T_{w_0 \hat{w} w_0 }) x^{-w_0 c}.
\end{equation*}
Taking the $\gamma_{w_0 \tilde w^{c} w_0 , -w_0c}^{m}$-coefficient of both sides, applying $w_0 \iota$ to the result, and using Theorem \ref{thm:met_minus_main} gives the result.
\end{proof}

Recall $\widetilde{C}_{\Lambda^m}^0$ is a fixed choice of representatives for $C_{\Lambda^m}^0 / \Omega_{\Lambda^m}$. For $\theta \in W\cdot \widetilde{C}^{0}_{\Lambda^m}$ and $\mu \in \Lambda^{++}$, recall the Whittaker functions $\tilde{\phi}_{\theta}^{o}(\mathbf{y}; \varpi^{\rho' - \mu}; v)$ as defined in \cite{BBBG21} and discussed in Section \ref{sec:Whittaker}.  We now use the results of this section to obtain formulas for $\tilde{\phi}_{\theta}^{o}$ in arbitrary types.  

\begin{theorem}\label{thm:Whitt_arb}
Set $v = q^{2}$.  Let $\mu - \rho \in \Lambda^{+}$ with decomposition $w_0 \mu = \eta + \hat{w} c$ for some $\eta \in \Lambda^{m}$, $\hat{w} \in W^{c}$, $c \in \widetilde{C}^{0}_{\Lambda^m}$ and let $\theta \in W \cdot \widetilde{C}^{0}_{\Lambda^m}$.   Also let $\rho'$ be a fixed element in $\Lambda$ such that $\rho - \rho'$ is constant.  Then $w_0 (\mathbf{y}^{\rho' - \theta} \tilde{\phi}_{\theta}^{o}(\mathbf{y}; \varpi^{\rho' - \mu}; v))$ is a polynomial in $\mathbb{F}[\Lambda^m]$, and it vanishes unless $w_0 \theta = \widetilde{w} c$ for some $w_0 \widetilde{w} \in W^{c}$, and in this case,
\begin{multline*}
w_0 \Big( \mathbf{y}^{\rho' - \theta} \tilde{\phi}_{\theta}^{o}(\mathbf{y}; \varpi^{\rho' - \mu}; v) \Big) \\= \frac{v^{l(w_0)}(-1)^{l( w_0 \hat{w} w_0)} F_{\mathbf{1}^{-}}^{-w_0 c}( w_0 \tilde w^{c} w_0) }{h^{m, -w_0 c}_{v}(w_0 \hat{w} w_0 )}  \cdot \pi^{m}_{v} \Big( T_{ \tilde w^{c} w_0 } \mathbf{1}^{-}_{W_{-w_0 c}} T_{w_0 \hat{w} w_0}^{-1} \Big)  \mathbf{y}^{w_0 \eta }. 
\end{multline*}
\end{theorem}

\begin{proof}
Follows from Proposition \ref{thm:phi_to_gamma} and Corollary \ref{cor:coeffs_w0}.
\end{proof}

We now specialize the previous theorem to $GL_{r}$ to prove Theorem \ref{cor:Whitt_typeA}.  In particular, set $m = n$ and take $\widetilde{C}_{\Lambda^m}^0$ as in \eqref{eq:glr_cosets}.  

\begin{proof}[Proof of Theorem \ref{cor:Whitt_typeA}]
We obtain the result from Theorem \ref{thm:Whitt_arb} in the $GL_r$ case as follows.  First, for $\mathbf{y}^{\rho_{\text{GL}} - \lfloor \theta \rfloor_{n}} \tilde{\phi}_{\theta}^{o}(\mathbf{y}; \varpi^{\rho_{\text{GL}} - \mu}; v)$ to be nonzero, we must have $\lfloor \theta \rfloor_{n} = w_0 \widetilde{w} c$ for some $w_0 \widetilde{w} \in W^{c}$, where $ \mu =w_0 \eta + w_0 \hat{w} c$ (for some $\eta \in \Lambda^{m}$, $c \in \widetilde{C}_{\Lambda^m}^0$, $\hat{w} \in W^{c})$.  This is equivalent to requiring that $\mu \text{ mod }n $ is a permutation of $\theta$.

For the second part, note that $-\theta = n^{r} - w_0 \widetilde{w} c \mod n = w_0 \widetilde{w}^{c} w_0 (n^{r} - w_0 c) \mod n$, where we have used Lemma \ref{lem:w0_conj}.  Set $w := \widetilde{w}^{c} w_0$ and $\textbf{c} := n^{r} - w_0 c$, and note that since $c \in \widetilde{C}_{\Lambda^m}^0$, $\textbf{c}$ is a decreasing $r$-tuple of elements in $\{1, \dots, n\}$.  So we have $-w_0 \theta = w \textbf{c} \mod n$.  Similarly, we have $-\mu = n^{r} - (w_0 \hat{w} w_0) w_0 c  \mod n$, since $-w_0 \eta \in \Lambda^{m}$.  Set $w' := w_0 \hat{w} w_0$, so that $-\mu = n^{r} - w' w_0 c \mod n = w'(n^{r} - w_0 c) \mod n = w' \textbf{c} \mod  n$.  Finally, by Lemma \ref{lem:w0_conj}, $ w', w_0 w \in W^{\mathbf{c}}$.

We use Proposition \ref{prop:pim_restr_pi}, Remark \ref{met_nonmet_v}, Proposition \ref{prop:TDW_trans} and  Proposition \ref{prop:TDW_antisymm} to relate the operators $\pi^{m}_{v}(T_w)$ to the operators $\mathcal{T}_{w, v^{1/2}}$ on the subspace $\mathbb{F}[\Lambda]$, noting that $w_0 \eta \in \Lambda^m$.  In particular, this gives the rescaling of variables $\mathbf{y} \rightarrow \mathbf{y}^{1/n}$ in the LHS, and $w_0 \eta \rightarrow \frac{w_0 \eta}{n} = \lambda + \rho_{GL}$ in the RHS.

Note also that $W_{\textbf{c}} = W_{-w_0 c}$, and $w_0 \eta = n \lfloor \frac{\mu}{n}\rfloor = n(\lambda + \rho_{GL})$.  Finally expressing the coefficient in terms of the data $\mathbf{c}, w, w'$ and simplifying gives the first equality.

For the second equality, note first that $\mathcal{T}_{w, v} = \mathcal{T}_{w^{\textbf{c}},v} \mathcal{T}_{w_{\textbf{c}},v}$.  Moreover, we have
\begin{equation*}
\mathcal{T}_{w_{ \textbf{c}},v} \sum_{u \in W_{ \textbf{c}}} \mathcal{T}_{u,v} = v^{-2l(w_{\textbf{c}})} \cdot \sum_{u \in W_{ \textbf{c}}} \mathcal{T}_{u,v},
\end{equation*}
by Propositions \ref{prop:TDW_trans}, \ref{prop:TDW_antisymm} and \eqref{T1}.  Using the fact that $\mathcal{T}_{w^{\textbf{c}},v} \mathcal{T}_{u,v} = \mathcal{T}_{w^{ \textbf{c}} u,v}$ for $u \in W_{ \textbf{c}}$, the desired expression is equal to 
\begin{equation*}
C v^{-l(w_{\mathbf{c}})} \cdot \sum_{u \in W_{\mathbf{c}}} \mathcal{T}_{w^{c} u, v^{1/2}} \mathcal{T}_{w', v^{1/2}}^{-1} \mathbf{y}^{n(\lambda + \rho_{GL})}.
\end{equation*}
The result now follows from Theorem \ref{thm:parahoric_to_T}.
\end{proof}

\begin{remark}\label{rem:met_special} The special cases where $\theta$ has all distinct parts or all equal parts (resp.) recovers \cite{BBBG21} Theorem D and E (resp.).\footnote{Note there is a typo in \cite{BBBG21} eq. (4.12) which gives a formula for $C$ in the distinct parts case: $w$ and $w'$ should be interchanged.}
  
Indeed, first suppose that $\theta$ has all distinct parts.  Then $W_{\textbf{c}} = 1$ and $W^{\textbf{c}} = W$, so there is no summation in the first equality and $w^{c} = w$, $w_{c} = 1$.  Moreover, $\psi^{\emptyset}_{w}(\mathbf{z}; \varpi^{-\lambda} w') = \phi_{w}(\mathbf{z}; \varpi^{-\lambda} w')$, which is a non-metaplectic Iwahori Whittaker function, where $\mathbf{z} := \mathbf{y}^{n}$.  In this case, one computes
\begin{equation*}
C' = \frac{(-1)^{l(w_0)} v^{l(w)/2 + l(w')/2 + l(w_0)/2} h_{v}^{m, \mathbf{c}}(w_0 w)}{h_{v}^{m, \mathbf{c}}(w')}.
\end{equation*}

Now suppose $\theta$ has all equal parts.  Then $W_{\textbf{c}} = W$ and $W^{\textbf{c}} = 1$.  So $w = w_0$, $w^{c} = 1$, $w_{c} = w_0$ and $w' = 1$.  In this case, we obtain a non-metaplectic spherical Whittaker function.  Indeed, by Propositions \ref{prop:TDW_trans} and \ref{prop:TDW_antisymm}
 and Lemma \ref{lem:antisym_w}, we have
\begin{align*}
(-1)^{l(w_0)} v^{2l(w_0)} \cdot \mathcal{T}_{w_0, v^{1/2}} \sum_{u \in W} \mathcal{T}_{u, v^{1/2}} \mathbf{z}^{\lambda + \rho_{GL}} &=  v^{2l(w_0)} (-v^{1/2})^{-l(w_0)} \pi_{v}(T_{w_0} \mathbf{1}^{-})  \mathbf{z}^{w_0(\lambda + \rho_{GL})} \\
&=v^{l(w_0)} \sum_{w \in W} \mathcal{T}_{w, v^{1/2}} \mathbf{z}^{w_0(\lambda + \rho_{GL})}.
\end{align*}
By \eqref{eq:sphW_nonmet_typeA}, this is equal to $\mathbf{z}^{w_0 \rho_{GL}} \widetilde{\mathcal{W}}_{\lambda}(\mathbf{z};v)$, as desired.
\end{remark}

\begin{remark}
In Remark 4.12 of \cite{BBBG21}, a more general duality conjecture is discussed, involving metaplectic \textit{parahoric} Whittaker functions instead of metaplectic \textit{spherical} Whittaker functions.  It would be interesting to see if this problem is amenable to the techniques developed in this paper.
\end{remark}

%%%%%%%%%%%%%%%%%%%%%%%%%%%%

\section{(Anti-)symmetric quasi-polynomials}\label{sec:symm_qps}

In \cite{SSV2}, we defined a subspace of (anti-)symmetric quasi-polynomials in terms of the quasi-polynomial representation $\pi^{qp}$.  In this section, we show that (anti-)symmetric quasi-polynomials can be identified with partially (anti)-symmetric polynomials.  Analogous results hold in the metaplectic context, but are omitted for brevity.

We first recall the definition of (anti-)symmetric quasi-polynomials.  We use $\epsilon$-notation to treat these cases uniformly, with $\epsilon = 1$ for the symmetric case and $\epsilon = -1$ for the anti-symmetric case. 

\begin{definition}\label{def:qp_symm}
The quasi-polynomial $f \in \mathbb{F}[E]$ is $\epsilon$-symmetric if and only if $\pi^{qp}(T_i) f = \epsilon t_i^{\epsilon} f$ for all $1 \leq i \leq r$. 
\end{definition}

In the symmetric case, that is $\epsilon = 1$, this is equivalent to $\sigma^{qp}(s_i) f = f$ for all $1 \leq i \leq r$.  The resulting quasi-polynomials are symmetric with respect to a generalized Weyl group action, denoted $\sigma^{qp}$, which is related to an action of Chinta and Gunnells in the metaplectic context (see \cite{CG07, CG}).  

\begin{proposition}\label{prop:qp_symm_equiv}
The quasi-polynomial $f \in \mathbb{F}[E]$ is $\epsilon$-symmetric if and only if $f = \pi^{qp}(\mathbf{1}^{\pm})g$ for some $g \in \mathbb{F}[\widetilde{\mathcal{O}}_{\Lambda, c}]$.
\end{proposition}

\begin{proof}
If $f \in \mathbb{F}[E]$ is $\epsilon$-symmetric, using Definition \ref{def:symm_antisymm}, we have $\pi^{qp}(\mathbf{1}^{\pm})  f \propto f$.  The other direction follows from \eqref{T1}.
\end{proof}

Recall from Section \ref{sec:partial_symm} the definition of $J$-partially $\epsilon$-symmetric elements in $\mathbb{F}[\Lambda]$ (Definition \ref{def:Jsymm}).  We will show that $\epsilon$-symmetric quasi-polynomials in $\mathbb{F}[\widetilde{\mathcal{O}}_c]$ are in bijection with $J_c$-partially $\epsilon$-symmetric polynomials in $\mathbb{F}[\Lambda]$.

As in Section \ref{sec:qp}, we assume $\Lambda \in \mathcal{L}$ and $c \in C^{0}_{\Lambda}$.  So, $\mathbb{F}[\widetilde{\mathcal{O}}_{\Lambda, c}] \subseteq \mathbb{F}[E]$ is a free $\mathbb{F}[\Lambda]$-module with basis $x^{w c}$ for $w \in W^{c}$.  Let $w_0^{c} \in W^{c}$ and ${w_{0}}_{c} \in W_{c}$ be as in Definition \ref{def:wc_decomp}.  
Consider the $\pi^{qp}$-invariant subspace $\mathbb{F}[\widetilde{\mathcal{O}}_{\Lambda, c}] \subseteq \mathbb{F}[E]$.   Recall that for $f \in  \mathbb{F}[\widetilde{\mathcal{O}}_{\Lambda, c}] $, we have the decomposition
\begin{equation*}
f = \sum_{w \in W^{c}} \gamma_{w}^{qp}(f) x^{w c},
\end{equation*}
where $\gamma_{w}^{qp}(f) \in \mathbb{F}[\Lambda]$.  We give a characterization of the coefficients $\gamma_{w}^{qp}(f)$ when $f$ is $\epsilon$-symmetric.

\begin{theorem}\label{thm:qp_symm}
Suppose $f \in \mathbb{F}[\widetilde{\mathcal{O}}_{\Lambda, c}] $.  Then $f$ is $\epsilon$-symmetric if and only if both of the following conditions hold:
\begin{enumerate}
\item \begin{equation*}
\begin{cases}
w_0 \gamma_{w_0^{c}}^{qp}(f) \text{ is } J_c\text{-partially symmetric}, & \epsilon = 1 \\
 \iota w_0 \gamma_{w_0^{c}}^{qp}(f) \text{ is } J_c\text{-partially anti-symmetric}, & \epsilon = -1
\end{cases}
\end{equation*}

\item 
\begin{equation*}
\gamma_{w}^{qp}(f) = 
\begin{cases}
 t({w_0}_{c}) w_0 \pi(T_{(w_0 w)^{-1}}^{-1}) w_0 \gamma_{w_0^{c}}^{qp}(f)$ for $w \in W^{c}, & \epsilon = 1 \\
 t({w_0}_{c}) (-1)^{l(w_0) + l(w)} \iota w_0 \pi(T_{w_0 w}) \iota w_0 \gamma_{w_{0}^{c}}^{qp}(f), & \epsilon = -1.
 \end{cases}
\end{equation*}
\end{enumerate}
\end{theorem}
\begin{proof}
Let $f \in \mathbb{F}[\widetilde{\mathcal{O}}_{\Lambda, c}] $.  By Proposition \ref{prop:qp_symm_equiv}, $f$ is (anti-)symmetric if and only if $f = \pi^{qp}(\mathbf{1}^{\pm}) g$ for some $g \in  \mathbb{F}[\widetilde{\mathcal{O}}_{\Lambda, c}]$.  Moreover, by Theorem \ref{thm:XT_formula}, any $g \in  \mathbb{F}[\widetilde{\mathcal{O}}_{\Lambda, c}]$ can be expressed as $g = \pi^{qp}(h) x^{c}$ for some $h \in \widetilde{\mathcal{H}}_{\Lambda}$.

Suppose first that $f = \pi^{qp}(\mathbf{1}^{\pm} h) x^{c}$ for some $h \in \widetilde{\mathcal{H}}_{\Lambda}$.  By Corollary \ref{cor:symm_char}, we have
\begin{multline}\label{eq:qp_symm_coeffs}
\gamma_{w}^{qp}(f) = \gamma_{w}( \pi^{qp}(\mathbf{1}^{+} h) x^{c}) \\ = t(w_0) w_0  \pi \Big( T_{(w_0 w)^{-1}}^{-1} \mathbf{1}^{+}_{W_{c}} \Big)  \sum_{\hat{w} \in W^{c}} \pi(T_{\hat{w}^{-1}}) \Big( \sum_{u \in W_{c}} t(u) \gamma_{\hat{w} u}(h) \Big)
\end{multline}
for $w \in W^{c}$ if $\epsilon = 1$, and by Corollary \ref{cor:antisymm_char}, we have
\begin{multline}\label{eq:qp_antisymm_coeffs}
\gamma_{w}^{qp}(f) = \gamma_{w}( \pi^{qp}(\mathbf{1}^{-} h) x^{c})  = t(w_0)^{-1} t(w_0(W_{c}))^{2} \cdot \\
\cdot \sum_{\hat{w} \in W^{c}} (-1)^{l(\hat{w}) + l(w)} \iota w_0 \pi \Big( T_{w_0 w} \mathbf{1}^{-}_{W_{c}} T_{\hat{w}}^{-1} \Big) \iota \Big( \sum_{u \in W_{c}} t(u) \gamma_{\hat{w} u}(h) \Big)
\end{multline}
for $w \in W^{c}$ if $\epsilon = -1$.
Setting $w = w_0^{c}$ and noting that
\begin{align*}
T_{(w_0 w_{0}^{c})^{-1}}^{-1} \mathbf{1}^{+}_{W_{c}} &= T_{{w_0}_{c}}^{-1} \mathbf{1}^{+}_{W_{c}}  =  t({w_0}_c)^{-1} \mathbf{1}^{+}_{W_{c}} \\
T_{w_0 w_{0}^{c}} \mathbf{1}_{W_{c}}^{-} &= T_{{{w_0}_c}^{-1}} \mathbf{1}^{-}_{W_{c}} = (-1)^{l({w_0}_c)}t({w_0}_c)^{-1} \mathbf{1}^{-}_{W_{c}}
\end{align*}
gives condition (1), and the second follows by relating $\gamma_{w}^{qp}(f)$ to $ \gamma_{w_0^{c}}^{qp}(f)$ using \eqref{eq:qp_symm_coeffs} and \eqref{eq:qp_antisymm_coeffs}.

Now suppose $\gamma_{w_0^{c}}^{qp}(f) = w_0 \pi( \mathbf{1}^{+}_{W_{c}}) g$ for some $g \in \mathbb{F}[\Lambda]$, and 
\begin{equation*}
\gamma_{w}^{qp}(f) = t({w_0}_{c}) w_0 \pi(T_{(w_0 w)^{-1}}^{-1}) \pi( \mathbf{1}^{+}_{W_{c}}) g
\end{equation*}
for $w \in W^{c}$.  Then
\begin{align*}
f &= \sum_{w \in W^{c}} \gamma_{w}(f) x^{w c} \\
&= \sum_{w \in W^{c}} \Big[ t({w_0}_c) w_0 \pi(T_{(w_0 w)^{-1}}^{-1}) \pi( \mathbf{1}^{+}_{W_{c}}) g \Big] x^{w c}.
\end{align*}
By Corollary \ref{cor:symm_char} part (2) (with the special case where $h := g \in \mathbb{F}[\Lambda]$), this is equal to
\begin{equation*}
  \pi^{qp}\Big(\mathbf{1}^{+} \frac{t({w_0}_c) g}{t(w_0)} \Big) x^{c},
\end{equation*}
so $f$ is a symmetric quasi-polynomial. 

The argument in the anti-symmetric case is similar.  Suppose $\gamma_{w_0^{c}}^{qp}(f) = \iota w_0 \pi( \mathbf{1}^{-}_{W_{c}}) g$ for some $g \in \mathbb{F}[\Lambda]$, and 
\begin{equation*}
\gamma_{w}(f) =  t({w_0}_{c}) (-1)^{l(w_0) + l(w)} \iota w_0 \pi(T_{w_0 w}) \pi(\mathbf{1}^{-}_{W_c})g
\end{equation*}
for $w \in W^{c}$.  Then
\begin{align*}
f &= \sum_{w \in W^{c}} \gamma_{w}(f) x^{w c} \\
&= \sum_{w \in W^{c}} \Big[ t({w_0}_{c}) (-1)^{l(w_0) + l(w)} \iota w_0 \pi(T_{w_0 w}) \pi(\mathbf{1}^{-}_{W_c})g \Big] x^{w c}.
\end{align*}
By Corollary \ref{cor:antisymm_char} part (2) (with the special case where $h := \iota g \in \mathbb{F}[\Lambda]$), this is equal to
\begin{equation*}
  \pi^{qp}\Big(\mathbf{1}^{-} \frac{(-1)^{l(w_0)}t(w_0)t({w_0}_c) \iota g}{t(w_0(W_c))^{2}} \Big) x^{c},
\end{equation*}
so $f$ is an anti-symmetric quasi-polynomial. 
\end{proof}

\begin{remark}
\textit{Special cases of Theorem \ref{thm:qp_symm}.}
\begin{enumerate}
\item Suppose $J_c = [1,r]$, so $J_c$-partially symmetric polynomials are just (fully) symmetric polynomials.  Also $c=0$, $W_{c} = W$, $W^{c} = \{1 \}$ and $\widetilde{\mathcal{O}}_c = \Lambda$.  So, in this case, symmetric quasi-polynomials are just the usual symmetric polynomials.

\item Suppose $J_c = \emptyset$, so every polynomial in $\mathbb{F}[\Lambda]$ is $J_c$-partially symmetric.  Also $W_{c} = \{1 \}$ and $W^{c} = W$.  So, in this case, symmetric quasi-polynomials are of the form
\begin{equation*}
\sum_{w \in W} \Big[ w_0 \pi(T_{w^{-1}}^{-1}) p\Big] x^{w_0 w c},
\end{equation*}
where $p \in \mathbb{F}[\Lambda]$.
\end{enumerate}
\end{remark}

Let $\mathbb{F}[\widetilde{\mathcal{O}}_{\Lambda, c}]^{\epsilon-\text{sym}} = \pi^{qp}(\mathbf{1}^{\pm}) \mathbb{F}[\widetilde{\mathcal{O}}_{\Lambda, c}]$ denote the space of $\epsilon$-symmetric quasi-polynomials, and let $\mathbb{F}[\Lambda]^{\epsilon, J_c-\text{sym}}$ denote the space of $J_c$-partially $\epsilon$-symmetric polynomials.  We have the following restatement of Theorem \ref{thm:qp_symm}.

\begin{corollary}\label{cor:bijection}
There is a bijection $\phi_{c, \epsilon}: \mathbb{F}[\widetilde{\mathcal{O}}_{\Lambda, c}]^{\epsilon-\text{sym}} \rightarrow \mathbb{F}[\Lambda]^{\epsilon, J-\text{sym}}$ defined by
\begin{equation*}
\phi_{c, \epsilon}(f) = \frac{t({w_0}_c)}{t(w_0)}\begin{cases}
w_0 \gamma_{w_0^{c}}^{qp}(f), & \epsilon = 1 \\
 \iota w_0 \gamma_{w_0^{c}}^{qp}(f), & \epsilon = -1
\end{cases}
\end{equation*}
for $f \in  \mathbb{F}[\widetilde{\mathcal{O}}_{\Lambda, c}]^{\epsilon-\text{sym}}$.  Its inverse is given by
\begin{equation*}
\phi_{c, \epsilon}^{-1}(p) = t(w_0) \begin{cases}
\sum_{w \in W^{c}} \Big[ w_0 \pi(T_{(w_0 w)^{-1}}^{-1}) p \Big] x^{wc}, & \epsilon = 1 \\ 
\sum_{w \in W^{c}} \Big[ (-1)^{l(w_0) + l(w)} \iota w_0 \pi(T_{w_0 w}) p \Big] x^{wc}, & \epsilon = -1.
\end{cases}
\end{equation*}
for $p \in  \mathbb{F}[\Lambda]^{\epsilon, J-\text{sym}}$.
\end{corollary}

The results of this paper show that in the $q\rightarrow \infty$ limit, we have the following duality between symmetric quasi-polynomial generalizations of Macdonald polynomials of \cite{SSV2} and the Macdonald polynomials with prescribed partial symmetry of Definition \ref{def:MDptl}.

\begin{theorem}\label{thm:quasi_duality}
Let $\Lambda \in \mathcal L$, $c \in C_\Lambda^0$, $\mu \in \Lambda^+$ and $\hat w$ in $W^c$. Then we have
\begin{equation*}
\phi_{c, \pm}(p^{\pm}_{\mu + \hat w c}) = p^{J_{c}, \pm}_{\hat w^{-1} \mu}.
\end{equation*}
\end{theorem}

\begin{proof}
Follows from Definition \ref{def:MDptl} and Theorem \ref{cor:qp_main}.
\end{proof}

One might hope that this duality might extend directly to $q$-level as well, however this is not the case.

\begin{example}\label{ex:qlevel}
In the $GL_2$ setting, take $c = (1/2, 0)$ (so $J_c = \emptyset$), $\mu = (0, 0), \hat w = 1$. So we have $y = \mu + \hat w c = c$ and $\hat w^{-1} \mu = 0$. Using the formulas in \cite[Section 7.6]{SSV2} we have 
\begin{equation*}
\phi_{c, +}(\iota E^+_{-(\mu + \hat w c)}(q; t^{-1}; \tau)) = 1 + \frac{t^{-2} -1}{\tau^{\alpha_1^\vee} - 1}
\end{equation*}
On the other hand, we have
\begin{equation*}
	\iota E^{\emptyset, +}_{\hat w^{-1} \mu}(q; t^{-1}) = \iota E_{0}(q; t^{-1}) = 1.
\end{equation*}
Note that these do match in the $q \rightarrow \infty$ limit (under the regularity assumption $\lim_{q \rightarrow \infty} \tau^{-\alpha_1^\vee} = 0$), agreeing with Theorem \ref{thm:quasi_duality}.
\end{example}

\begin{remark}
It is an interesting open question to determine $\phi_{c,+}(\iota E^{+}_{-(\mu + \hat w c)}(q; \mathbf{t}^{-1}; \tau))$, i.e., which $J$-partially symmetric polynomials correspond to the symmetric quasi-polynomial generalizations of Macdonald polynomials introduced in \cite{SSV2}.
\end{remark}

%%%%%%%%%%%%%%%%%%%%%%%%%%%%%%%%%%%%%%%%%%%%%%%

\end{document}